\documentclass{amsart}
\usepackage[utf8]{inputenc}
\usepackage{microtype}
\usepackage{graphicx,enumerate,tikz-cd}
\usepackage [english]{babel}
\usepackage{pinlabel}
\usepackage{caption,subcaption}
\usepackage{amsmath,amssymb,amsthm,setspace}
\usepackage{accents}
\usepackage{hyperref}
\usepackage{csquotes}
\usepackage{mathrsfs} 
\newcommand{\Zm}[1]{\Z/#1\Z}
\newcommand{\symm}[1]{S_{#1}}

\title{Topological complexity of unordered configuration spaces of certain graphs}

\author{Steven Scheirer}

\address{Department of Mathematics and Computer Science\\Ashland University, Ashland, OH  44805}
\curraddr{Department of Mathematics and Statistics, Carleton College, Northfield MN 55057}
\email{sscheirer@carleton.edu}

\newcommand{\Z}{\mathbb{Z}}

\newcommand{\TC}{\mathrm{TC}}

\newcommand{\C}[2]{C^{#1}({#2})}
\newcommand{\UC}[2]{\mathrm{UC}^{#1}({#2})}
\newcommand{\D}[2]{D^{#1}({#2})}
\newcommand{\UD}[2]{\mathrm{UD}^{#1}({#2})}
\newcommand{\sym}[1]{S_{#1}}
\newcommand{\cycle}{\mathscr{C}}

\newcommand{\image}{\mathrm{image}}

\newcommand{\w}{\widetilde}

\numberwithin{equation}{section}
\newtheorem{thm}[equation]{Theorem}

\newtheorem{lemma}[equation]{Lemma}
\newtheorem{defn}[equation]{Definition}

\newtheorem{cor}[equation]{Corollary}

\def\from{\colon\thinspace}

\begin{document}

\keywords{Topological complexity, configuration spaces, graphs}
\subjclass[2010]{55R80, 57M15}
\begin{abstract} 
The unordered configuration space of $n$ points on a graph $\Gamma,$ denoted here by $\UC{n}{\Gamma},$ can be viewed as the space of all configurations of $n$ unlabeled robots on a system of one-dimensional tracks, which is interpreted as a graph $\Gamma.$  The topology of these spaces is related to the number of vertices of degree greater than 2; this number is denoted by $m(\Gamma).$  We discuss a combinatorial approach to compute the topological complexity of a ``discretized" version of this space, $\UD{n}{\Gamma},$ and give results for certain classes of graphs.  In the first case, we show that for a large class of graphs, as long as the number of robots is at least $2m(\Gamma)$, then $\TC(\UD{n}{\Gamma})=2m(\Gamma)+1.$  In the second, we show that as long as the number of robots is at most half the number of vertex-disjoint cycles in $\Gamma,$ we have $\TC(\UD{n}{\Gamma})=2n+1.$
\end{abstract}

\maketitle

\section{Introduction}\label{sec:intro}
For any path-connected space $X,$ let $P(X)$ denote the space of all continuous paths in $X;$ an element of $P(X)$ is a map $\sigma\from I\to X,$ where $I$ denotes the unit interval.  This gives a fibration $p\from P(X)\to X\times X$ which sends a path $\sigma$ to its endpoints:  $p(\sigma)=(\sigma(0),\sigma(1)).$  A (not necessarily continuous) section $s$ of this fibration takes a pair of points as input and produces a path between those points.  That is, such a section can be viewed as a rule which assigns a path between any two points in $X.$  If $s$ is continuous at a point $(x,y)\in X\times X,$ then whenever $(x',y')$ is obtained by slightly perturbing the point $(x,y),$ the path $s(x',y')$ only varies slightly from the path $s(x,y).$  Unless the space $X$ is contractible, it is impossible to find a section which is continuous over all of $X\times X.$  The topological complexity of $X,$ introduced by Farber and denoted by $\TC(X),$ is a measure of this inability to find a globally continuous section.   Specifically, we have the following definition:

\begin{defn}{\rm\cite{FarberTC}}\label{defn:TC}
For any path-connected space $X,$ the topological complexity of $X,$ denoted by $\TC(X),$ is the smallest integer $k$ such that there are open sets $U_1,\dots,U_k\subset X\times X$ which cover $X\times X,$ and continuous sections $s_i\from U_i\to P(X).$  If there is no such $k,$ let $\TC(X)=\infty.$
\end{defn}

Some authors prefer a ``reduced" version of topological complexity, which is one less than the number $k$ in Definition \ref{defn:TC}.  Viewing the sections $s_i$ as ``continuous rules," we can interpret the topological complexity of $X$ as the minimum number of continuous rules needed to describe how to move between any two points in $X.$  In light of this interpretation, it is often desirable to view the space $X$ as the configuration space of a robot or several robots, so that $\TC(X)$ is related to the problem of moving the robot(s) from one configuration to another.  One such example is the case in which $X$ is the space of configurations of $n$ robots which move along a system of one-dimensional tracks.  In this setting, a point in $X$ is a configuration of all $n$ robots, and studying the topological complexity of $X$ addresses the problem of moving all robots from one configuration to another.  The tracks are interpreted as a graph $\Gamma$ and the robots are then $n$ distinct points on $\Gamma.$  

Specifically, for any graph $\Gamma,$ let $\C{n}{\Gamma}$ denote the space of all configurations of $n$ distinct points on $\Gamma.$  That is, 
\[
\C{n}{\Gamma}=\overbrace{\Gamma\times\cdots\times\Gamma}^{n\ \text{times}}-\Delta,
\]
where $\Delta=\{(x_1,\dots,x_n):x_i=x_j\ \text{for some } i\ne j\}\subset\Gamma\times\cdots\times\Gamma.$  The space $\C{n}{\Gamma}$ will be called the \emph{ordered topological configuration space}.  We can view the points $x_1$ through $x_n$ as the locations of $n$ distinguishable robots on $\Gamma.$  In this space, the robots (viewed as infinitesimally small) are permitted to move arbitrarily close to one another.  Aside from the real-world impracticality of this, the space $\C{n}{\Gamma}$ faces another downfall.  The space $\Gamma\times\cdots\times\Gamma$ inherits a closed cell-structure from $\Gamma.$ The 0-cells of $\Gamma$ are the vertices, and the open 1-cells are the interiors of the edges.  Note that when we refer to an ``edge'' here, we mean the closure of a 1-cell in $\Gamma.$  We will refer to the interior of an edge as an ``open edge.'' Open cells in $\Gamma\times\cdots\times\Gamma$ are then products of the form $c=c_1\times\cdots\times c_n,$ where each $c_i$ is either a vertex or an open edge of $\Gamma;$ the dimension of such a cell is the number of edges that appear in the product.  However, by removing $\Delta,$ we lose this structure.  This can be addressed by working with a ``discretized'' version of the space, namely
\[
\D{n}{\Gamma}=\overbrace{\Gamma\times\cdots\times\Gamma}^{n\ \text{times}}-\w{\Delta},
\]
where $\w{\Delta}=\{c_1\times\cdots\times c_n:\overline{c}_i\cap\overline{c}_j\ne\emptyset\ \text{for some } i\ne j\}.$  Again, here each $c_i$ is either a vertex or open edge of $\Gamma.$  The space $\D{n}{\Gamma}$ will be called the \emph{ordered discrete configuration space}.  Since we remove entire open cells rather than just the subspace $\Delta,$ the space $\D{n}{\Gamma}$ retains the cell structure from $\Gamma\times\cdots\times\Gamma.$  Points in $\D{n}{\Gamma}$ are again viewed as the locations of $n$ robots on $\Gamma,$ but now any two robots must be separated by at least a full open edge of the graph $\Gamma.$  Open cells in $\D{n}{\Gamma}$ are now of the form $c=c_1\times\cdots\times c_n$ with the property that $\overline{c}_i\cap\overline{c}_j=\emptyset$ whenever $i\ne j.$

In both $\C{n}{\Gamma}$ and $\D{n}{\Gamma},$ the configurations are ordered in the sense that the specific robot at each specified location on $\Gamma$ is of importance.  That is, we view the robots as being labeled.  If we wish to ignore the labels of the robots, we can work with the \emph{unordered} configuration spaces, which are obtained by factoring out the action of the symmetric group $\sym{n}:$
\[ 
\UC{n}{\Gamma}=\C{n}{\Gamma}/\sym{n} \quad\quad \UD{n}{\Gamma}=\D{n}{\Gamma}/\sym{n}.
\]
The spaces $\UC{n}{\Gamma}$ and $\UD{n}{\Gamma}$ will be called the \emph{unordered} topological and discrete configuration spaces, respectively.  Since the order of the robots on $\Gamma$ is irrelevant in these spaces, a point in $\UC{n}{\Gamma}$ is viewed as a collection $\{x_1,\dots,x_n\}$ of points on $\Gamma$ with $x_i\ne x_j$ when $i\ne j.$  Likewise, points in $\UD{n}{\Gamma}$ are viewed as collections $\{x_1,\dots,x_n\}$ of points on $\Gamma$ such that $x_i$ and $x_j$ are separated by at least a full open edge of $\Gamma$ whenever $i\ne j.$  Cells in $\UD{n}{\Gamma}$ are viewed as collections $c=\{c_1,\dots,c_n\}$ where each $c_i$ is either a vertex or an open edge of $\Gamma,$ and $\overline{c}_i\cap\overline{c_j}=\emptyset$ whenever $i\ne j.$  Again, the dimension of such a cell is the number of edges that appear in the collection.

One may wonder if there is a relationship between the (un)ordered topological configuration space and the (un)ordered discrete configuration space.  Certainly in general, $\C{n}{\Gamma}$ and $\UC{n}{\Gamma}$ can be quite different than $\D{n}{\Gamma}$ and $\UD{n}{\Gamma},$ respectively, but Abrams shows that under suitable hypotheses on the subdivision of $\Gamma,$ the topological configuration spaces deformation retract onto the discrete configuration spaces.

\begin{thm}{\rm\cite{AbramsThesis,KimKoPark}}\label{thm:SuffSubdivision}
Let $\Gamma$ be a graph with at least $n$ vertices and suppose $\Gamma$ has the following properties:
\begin{enumerate}
\item Each path between distinct vertices of degree not equal to 2 in $\Gamma$ contains at least $n-1$ edges.
\item Each loop at a vertex in $\Gamma$ which is not homotopic to a constant map contains at least $n+1$ edges.
\end{enumerate}
Then $\C{n}{\Gamma}$ and $\UC{n}{\Gamma}$ deformation retract onto $\D{n}{\Gamma}$ and $\UD{n}{\Gamma},$ respectively.
\end{thm}
A graph that satisfies the conditions in Theorem \ref{thm:SuffSubdivision} is called \emph{sufficiently subdivided for $n$}.  Thus, provided the graph $\Gamma$ is sufficiently subdivided for $n,$ the topological features of the spaces $\C{n}{\Gamma}$ and $\UC{n}{\Gamma}$ can be studied using the cell structure of $\D{n}{\Gamma}$ and $\UD{n}{\Gamma},$ respectively.  From here on, we will assume all graphs are sufficiently subdivided.

It is not surprising that many topological features of the configuration spaces of graphs are related to the number of vertices of degree at least three.  These vertices are called \emph{essential} vertices, and the number of essential vertices of $\Gamma$ is denoted by $m(\Gamma).$  The theorems listed below give the main known results regarding the topological complexity of graph configurations spaces in both the ordered and the unordered contexts.  The first result, due to Farber, addresses ordered configuration spaces of trees:

\begin{thm}\rm{\cite{FarberCollisionFree, FarberConfigSpacesMotionPlanning}}\label{thm:FarberTC}
Let $\Gamma$ be a tree with at least one essential vertex and let $n$ be an integer satisfying $n\ge2m(\Gamma).$  If $n=2,$ assume further that $\Gamma$ is not homeomorphic to the letter $Y.$  Then $\TC(\C{n}{\Gamma})=2\cdot m(\Gamma)+1.$
\end{thm}

In the exceptional case in which $\Gamma$ is homeomorphic to the letter $Y$ and $n=2,$ it is well known that $\C{n}{\Gamma}$ is homotopy equivalent the circle $S^1$ and thus $\TC(\C{n}{\Gamma})=\TC(S^1)=2.$  The author extended this result to include the unordered configuration spaces as well as values of $n$ which are less than $2 m(\Gamma)$:

\begin{thm}{\rm\cite{ScheirerTrees}}\label{thm:TCofTrees}
Let $\Gamma$ be a tree with at least one essential vertex.  Provided $n$ satisfies some technical restrictions which depend on properties of $\Gamma,$ we have 
\[
\TC(\C{n}{\Gamma})=\TC(\UC{n}{\Gamma})=2K+1,
\]
where $K=\min\left\{\left\lfloor\frac{n}{2}\right\rfloor,m(\Gamma)\right\}.$
\end{thm}
The complete description of the restrictions on $n$ mentioned in the above result is given in \cite{ScheirerTrees} and eliminates at most finitely many values of $n$ for any fixed tree $\Gamma.$  In fact, if the tree $\Gamma$ does not contain any vertices of degree 3, then there are no restrictions, and Theorem \ref{thm:TCofTrees} determines the topological complexity for all values of $n,$ provided the configuration spaces are connected.  It is worth noting the method for obtaining Theorem \ref{thm:TCofTrees} focuses primarily on the unordered spaces; the statement concerning the ordered spaces is essentially a byproduct. 

L\"utgehetmann and Recio-Mitter extended Farber's result concerning the ordered spaces to include all values of $n$:

\begin{thm}{\rm\cite{RecioMitter}}
\label{thm:RecioMitter}
Let $\Gamma$ be a tree with at least one essential vertex.  For $n=2,$ assume $\Gamma$ is not homeomorphic to the letter $Y$.  Then, 
\[
\TC(\C{n}{\Gamma})=2K+1,
\]
where $K=\min\left\{\left\lfloor\frac{n}{2}\right\rfloor,m(\Gamma)\right\}.$
\end{thm}

The results above all address the case in which the underlying graph is a tree.  L\"utgehetmann and Recio-Mitter also study the topological complexity of ordered configuration spaces of graphs which are not trees.  Their results address two classes of graphs: fully articulated graphs and banana graphs.  A graph is fully articulated if removing any essential vertex disconnects the graph.  The class of fully articulated graphs contains the class of trees as a subset.  The banana graph $B_k$ consists of two vertices $x$ and $y,$ and $k$ edges each having $x$ and $y$ as endpoints.  Note that no banana graph is fully-articulated.

\begin{thm}{\rm\cite{RecioMitter}}\label{thm:TCFullyArticulated}
If $\Gamma$ is fully articulated and $\Gamma$ is not homeomorphic to the letter $Y,$ and $n\ge 2m(\Gamma),$ then 
\[
\TC(\C{n}{\Gamma})=2m(\Gamma)+1.
\]
\end{thm}

\begin{thm}{\rm\cite{RecioMitter}}\label{thm:TCBanana}
The topological complexity of the banana graph $B_k$ is given by 
\[
\TC(\C{n}{B_k})=\begin{cases}
5,&\text{if } k\ge4\text{ and } n\ge3\\
3,&\text{if } k\ge3\text{ and } n\le2\text{ or }k=n=3.
\end{cases}
\]
\end{thm}

It should also be noted that in \cite{Guzman}, Aguilar-Guzm\'{a}n, Gonz\'{a}lez, and Hoekstra-Mendoza study $\TC_s(\C{n}{\Gamma}),$ where $\Gamma$ is a tree and $\TC_s(X)$ stands for the \textit{higher topological complexity} of $X.$  The higher topological complexity of $X$ is a generalization of $\TC(X),$ where $\TC(X)=\TC_2(X);$ we will omit the precise definition here as we only work with $\TC(X).$

In this paper, we generalize the approach taken in \cite{ScheirerTrees} to study $\TC(\UC{n}{\Gamma})$ for graphs $\Gamma$ that are not trees.  Our main results are given in Theorems \ref{thm:S_Graph} and \ref{thm:VertexDisjointUnordered} below.  Theorem \ref{thm:S_Graph} refers to a class of graphs called $S$-\textit{graphs.}  The precise definition of an $S$-graph is given in Definition \ref{defn:S_Graph}, but it is worth noting that the class of $S$-graphs is a (proper) subclass of the class of fully articulated graphs (see the discussion following Definition \ref{defn:S_Graph}), so Theorem \ref{thm:S_Graph} partially extends Theorem \ref{thm:TCFullyArticulated} to the unordered setting.  In Theorem \ref{thm:VertexDisjointUnordered}, $\nu(\Gamma)$ refers to the maximum cardinality of a collection of vertex-disjoint cycles in $\Gamma.$

\begin{thm}\label{thm:S_Graph}
Let $\Gamma$ be an $S$-graph, and let $n$ be an integer satisfying $n\ge 2m(\Gamma).$  Then
$\TC(\UD{n}{\Gamma})=2m(\Gamma)+1.$
\end{thm}

\begin{thm}\label{thm:VertexDisjointUnordered}
Let $\Gamma$ be any graph with $\nu=\nu(\Gamma)\ge 2,$ and let $n$ be an integer satisfying $1\le n\le\frac{1}{2}\nu.$  Then $\TC(\UD{n}{\Gamma})=2n+1.$
\end{thm}

We also give an analogous version of Theorem \ref{thm:VertexDisjointUnordered} in the ordered setting in Theorem \ref{thm:VertexDisjointOrdered}.  It is also worth noting that the results listed here may seem to suggest that the topological complexity of the unordered configuration space $\UC{n}{\Gamma}$ agrees with that of the ordered configuration space $\C{n}{\Gamma}.$  However, as mentioned in \cite{RecioMitter}, one can construct an infinite family of trees $\Gamma$ such that $\TC(\UC{n}{\Gamma})<\TC(\C{n}{\Gamma})$ for certain integers $n.$  At present, a more precise description of the relationship between the topological complexities in the ordered and unordered context remains an open problem.

The rest of this paper is organized as follows.  In Section \ref{sec:TC}, we recall the main tools for computing topological complexity and provide an example of how these may be used to compute topological complexity of certain ordered graph configuration spaces. In Section \ref{sec:DMT}, we include a description of an approach to studying the unordered configuration space $\UD{n}{\Gamma}$ given by Daniel Farley and Lucas Sabalka using Forman's discrete Morse theory and generalize some of their results which will be used in the proofs of Theorems \ref{thm:S_Graph} and \ref{thm:VertexDisjointUnordered} in Section \ref{sec:MainResults}.

The author would like to extend his gratitude to the anonymous reviewer for carefully reading the first version of this manuscript and providing very useful suggestions on the structure of the paper.

\section{Topological Complexity}\label{sec:TC}
The main tools we will use for computing topological complexity are given in the following theorem.

\begin{thm}{\rm\cite{FarberTC}}\label{thm:TCTools}
We have the following properties of topological complexity:
\begin{enumerate}
    \item If $X$ and $Y$ are homotopy equivalent, then $\TC(X)=\TC(Y).$
    \item If $X$ is a CW complex of dimension $k,$ then $\TC(X)\le2k+1.$
    \item $\TC(X)>\mathrm{zcl}(X),$ where $\mathrm{zcl}(X)$ denotes the zero-divisors cup length of $X.$ 
\end{enumerate}
\end{thm}
Recall that the zero-divisors cup length of $X$ is the largest integer $i$ such that there are elements 
\[
a_1,\dots,a_i\in\ker(\smallsmile\from H^*(X;\mathbf{k})\otimes H^*(X;\mathbf{k})\to H^*(X;\mathbf{k}))
\]
such that the product $a_1\cdots a_i$ is non-zero, and $\mathbf{k}$ is some coefficient field.  In this paper, we will take $\mathbf{k}=\Zm{2}$ and omit the coefficients from the notation.

As an illustration of how Theorem 2.1 will be used, we consider the case in which the graph $\Gamma$ contains enough cycles to allow $n$ ordered robots to simultaneously move around the cycles without crossing paths.  Recall a cycle in a graph is a sequence of vertices $\cycle=(v_1,v_2,\dots,v_k,v_1)$ in which $v_i\ne v_k$ for $i\ne k$ and there is an edge connecting $v_i$ to $v_{i+1}$ (where the indices are read modulo $k$).  Two cycles $\cycle_1$ and $\cycle_2$ are \emph{vertex-disjoint} if they have no vertices in common; we denote the maximum number of vertex-disjoint cycles in $\Gamma$ by $\nu(\Gamma).$  As an example, consider the complete graph on $m$ vertices, $K_m.$  Any cycle must contain at least 3 vertices, so $\nu(K_m)\le \left\lfloor\frac{m}{3}\right\rfloor.$  This upper bound can be realized by considering the smallest cycles containing distinct triples of vertices, so $\nu(K_m)= \left\lfloor\frac{m}{3}\right\rfloor.$

\begin{thm}\label{thm:VertexDisjointOrdered}
Let $\Gamma$ be any graph with $\nu=\nu(\Gamma)\ge 2,$ and let $n$ be an integer satisfying $2\le n\le\nu.$  Then $\TC(\D{n}{\Gamma})=2n+1.$
\end{thm}
\begin{proof}
The approach we take here is essentially the approach taken by Farber in \cite{FarberConfigSpacesMotionPlanning} and L\"{u}tgehetmann and Recio-Mitter in \cite{RecioMitter}.  Since $\D{n}{\Gamma}\subset \prod_{i=1}^n\Gamma,$ we have $\dim(\D{n}{\Gamma})\le n,$ so $\TC(\D{n}{\Gamma})\le 2n+1.$ 

For the lower bounds on $\TC(\D{n}{\Gamma})$, let $\cycle_1,\dots,\cycle_n$ be vertex-disjoint cycles in $\Gamma,$ and for each $i,$ fix a homeomorphism $\epsilon_i\from S^1\to\cycle_i.$
For each permutation $\sigma\in\symm{n},$ let 
\[
f_\sigma\from\prod_{i=1}^nS^1\to \prod_{i=1}^n\Gamma
\]
be the map which sends $(x_1,\dots,x_n)$ to $(\epsilon_{\sigma(1)}(x_1),\dots,\epsilon_{\sigma(n)}(x_n)).$  Since the cycles $\cycle_1,\dots\cycle_n$ are disjoint, this gives a map 
\[
\w{f}_\sigma\from\prod_{i=1}^nS^1\to\D{n}{\Gamma}.
\]
The image of $(x_1,\dots,x_n)$ is a configuration in which robot $i$ falls on cycle $\cycle_{\sigma(i)}.$  Let $\Phi\from \D{n}{\Gamma}\to\prod_{i=1}^n\Gamma$ denote the inclusion.  These maps fit into the following commutative diagram.

\[
\begin{tikzcd}
\displaystyle
\prod_{i=1}^nS^1\arrow[r,"\w{f}_\sigma"]\arrow[dr,"f_\sigma"]&\D{n}{\Gamma}\arrow[d,"\Phi"]\\
&\displaystyle\prod_{i=1}^n\Gamma
\end{tikzcd}
\]

For each $j=1,\dots,n,$ let $a_j$ denote a non-zero class in $H_1(\Gamma)$ corresponding to $\cycle_j$ and let $\alpha_j\in H^1(\Gamma)$ denote the cohomology class dual to $a_j.$ Let 
\[
\mu_i=\Phi^*(1\times\cdots\times\alpha_i\times\cdots\times1)\in H^1(\D{n}{\Gamma}),
\]
where the only nontrivial term, $\alpha_i,$ falls in the $i$th factor, and let 
\[
\mu'_i=\Phi^*(1\times\cdots\times\alpha_{i+1}\times\cdots\times1)\in H^1(\D{n}{\Gamma}),
\]
where the non-trivial term, $\alpha_{i+1},$ falls in the $i$th factor (subscripts are to be read modulo $n$).  The diagram shows $\mu_1\cdots\mu_n$ and $\mu_1'\cdots\mu_n'$ are non-zero in $H^*(\D{n}{\Gamma}).$

Let $\overline{\mu}_i=\mu_i\otimes1+1\otimes\mu_i$ and $\overline{\mu}'_i=\mu'_i\otimes1+1\otimes\mu'_i.$  Each $\overline{\mu}_i$ and $\overline{\mu}'_i$ is a zero-divisor.  Consider the following product:
\begin{equation}
\biggl(\prod_{i=1}^n\overline{\mu}_i\biggr)\cdot \biggl(\prod_{i=1}^n\overline{\mu}'_i\biggr).
\label{eqn:orderedZD}
\end{equation}
We wish to show this product is non-zero.  For $i=1,\dots,n,$ let $b_i\in H_1(\prod S^1)$ denote a non-zero class corresponding to the copy of $S^1$ in the $i$th factor, let $b=b_1\times\cdots\times b_n\in H_n(\prod S^1),$ and let $y=(\w{f}_\sigma)_*(b),$ where $\sigma(i)=i$ and let $y'=(\w{f}_{\sigma'})_*(b),$ where $\sigma'(i)=i+1.$  Notice that $(\Phi)_*(y)=a_1\times a_2\times\cdots\times a_{n-1}\times a_n$ and $(\Phi)_*(y')=a_2\times a_3\cdots\times a_n\times a_1.$

We will show that (\ref{eqn:orderedZD}) is non-trivial by showing it acts non-trivially on $y\otimes y'.$ The product in (\ref{eqn:orderedZD}) is a sum of terms of the form $A\otimes B$ where $A$ and $B$ are monomials of the form
\begin{equation}
\mu_{i_1}\cdots\mu_{i_s}\mu'_{j_{1}}\cdots\mu'_{j_t}.\label{eqn:Monomial}
\end{equation}

Assume $A\otimes B\ne0.$  Since the top dimension in $H^*(\D{n}{\Gamma})$ is $n$, we have $|A|=|B|=n,$ so $s+t=n.$  If $I=\{i_1,\dots,i_s\},$ we will write $\mu_{i_1}\cdots\mu_{i_s}$ as $\mu_I$ and define $\mu'_J$ similarly for $J=\{j_1,\dots,j_s\}$. We define $\mu_\emptyset=\mu'_\emptyset=1,$ so that each term in (\ref{eqn:orderedZD}) can be written as $\mu_I\mu'_J\otimes\mu_{\overline{I}}\mu'_{\overline{J}},$ where $\overline{I}$ and $\overline{J}$ denote the elements of $\{1,\dots,n\}$ which aren't contained in $I$ and $J,$ respectively.  If $|I\cup J|<n,$ then there must be two indices in (\ref{eqn:Monomial}) that share the same value, say, $i_0.$ In this case, the product $\mu_I\mu'_J$ contains $\Phi^*(\cdots\times a_{i_0}a_{i_0+1}\times\cdots)=0,$ so, we may assume each index is distinct, so that $I\cup J=\overline{I}\cup\overline{J}=\{1,\dots,n\}.$

Let $\delta_i=0$ if $i\in I$ and $\delta_i=1$ if $i\in J.$  We have
\begin{align*}
\langle \mu_I\mu'_J,y\rangle
&=\langle\Phi^*(\alpha_{1+\delta_1}\times\cdots\times\alpha_{n+\delta_n}),y\rangle\\ &=\langle\alpha_{1+\delta_1}\times\cdots\times\alpha_{n+\delta_n},\Phi_*(y)\rangle\\
&=\langle\alpha_{1+\delta_1}\times\cdots\times\alpha_{n+\delta_n},a_1\times\cdots\times a_n\rangle.
\end{align*}
This evaluation is nonzero if and only if $\delta_i=0$ for each $i,$ so that $I=\{1,\dots,n\}$ and $J=\emptyset.$  

Similarly, $\langle\mu_I\mu'_J,y'\rangle\ne 0$
if and only if $I=\emptyset$ and $J=\{1,\dots,n\},$ so 
\begin{align*}
&\biggl\langle \biggl(\prod_{i=1}^n\overline{\mu}_{i}\biggr)\cdot \biggl(\prod_{i=1}^n\overline{\mu}'_{i}\biggr),y\otimes y'\biggr\rangle
=\langle \mu_{1}\cdots\mu_{n}\otimes \mu'_{1}\cdots\mu'_{n},y\otimes y'\rangle
\ne0,
\end{align*}
establishing the lower bound $\TC(\D{n}{\Gamma})\ge 2n+1,$ and completing the proof.
\end{proof}

\section{Discrete Morse Theory and Unordered Configuration Spaces}\label{sec:DMT}
In this section, we introduce the machinery we will use to prove Theorems \ref{thm:S_Graph} and \ref{thm:VertexDisjointUnordered}.  In Subsection \ref{subsec:DMTBackground}, we give brief background information on discrete Morse theory.  In Subsection \ref{subsec:FSDGVF}, we describe Farley and Sabalka's approach to studying $\UD{n}{\Gamma}$ using discrete Morse theory and give straightforward upper bounds on $\TC(\UD{n}{\Gamma}).$  Finally, in Subsection \ref{subsec:Cohomology}, we extend results of Farley and Sabalka to give (co)homological information about the spaces $\UD{n}{\Gamma},$ which will be used to establish lower bounds in the proofs of the main results. 

\subsection{Background on discrete Morse theory}\label{subsec:DMTBackground}
We will only give a very brief overview of the main ideas of discrete Morse theory here and refer the reader to Forman's texts \cite{Forman} and \cite{FormanUser} for a detailed account. One of the main tools in discrete Morse theory is the notion of a \emph{discrete gradient vector field} on a cell complex $X.$  A discrete gradient vector field $W$ can be viewed as a way of assigning to each $i$-cell $\sigma$ either an $(i+1)$-cell $\tau$ of which $\sigma$ is a face or a ``neutral'' element $0$.  This assignment must satisfy a number of conditions which classify each cell of $X$ into exactly one of three types:
\begin{defn}\label{defn:CellTypes}
Let $W$ be a discrete gradient vector field on $X.$  We have the following classification of the cells of $X:$
\begin{enumerate}
\item If $W(\sigma)\ne0,$ then $\sigma$ is called a redundant cell.
\item If $\sigma\in\image(W),$ then $\sigma$ is called a collapsible cell.
\item If $W(\sigma)=0$ and $\sigma\notin\image(W),$ then $\sigma$ is called a critical cell.
\end{enumerate}
\end{defn}
The critical cells of $X$ carry the most important topological information.  For example, we have the following:
\begin{thm}{\rm\cite{Forman}}\label{thm:CriticalDimension}
If $m_p$ is the number of critical $p$-cells of $X,$ then $X$ is homotopy equivalent to a space consisting of $m_p$ $p$-cells for each $p=0,1,2,\dots.$  In particular, if $M$ is the maximum dimension of a critical cell in $X,$ then $X$ is homotopy equivalent to a space of dimension $M.$
\end{thm}

A discrete gradient vector field $W$ may be used to compute the homology of $X$ as follows.  Let $C_p(X)$ denote the free $\Zm{2}$-module generated by the $p$-cells of $X,$ and let $\partial\from C_p(X)\to C_{p-1}(X)$ denote the cellular boundary operator of $X.$  We denote the homology of the complex $(C_*(X),\partial)$ by $H_*(X).$  

We may linearly extend $W$ to a map $W\from C_p(X)\to C_{p+1}(X)$  and define a map 
\[
f\from C_p(X)\to C_p(X)
\] 
by $f=1+\partial W+W\partial$ (here, $1$ represents the identity map).  The map $f$ has the property that for any chain $c\in C_p(X),$ there is some positive integer $m$ such that $f^m(c)=f^{m+1}(c),$ so that $f^\infty$ is well-defined \cite{FarleyHomology}.  

Let $M_p(X)$ denote the free $\Zm{2}$-module on the critical $p$-cells of $X,$ and define $\w{\partial}\from M_p(X)\to M_{p-1}(X)$ by $\w{\partial}=\pi\partial f^\infty,$ where $\pi\from  C_p(X)\to M_p(X)$ is projection onto the critical $p$-cells.  This gives the \emph{Morse Complex} \cite{FarleyFiniteness}
\[
\begin{tikzcd}
\cdots\arrow[r]&M_{p+1}(X)\arrow[r,"\w{\partial}"]&M_{p}(X)\arrow[r,"\w{\partial}"]&M_{p-1}(X)\arrow[r]&\cdots.\label{complex:MorseComplex}
\end{tikzcd}
\]
We denote the homology of the complex $(M_*(X),\w{\partial})$ by $H^\mathcal{M}_*(X).$ 
We have the following result of Farley:

\begin{thm}{\rm\cite{FarleyFiniteness}}\label{thm:MorseComplex}
The map $f^\infty\from M_*(X)\to C_*(X)$ induces an isomorphism $H^\mathcal{M}_*(X)\cong H_*(X).$  For each critical cell $c,$ the chain $f^\infty(c)\in C_*(X)$ satisfies $f^\infty(c)=c+\text{collapsible cells}.$
\end{thm}
Define $F\from C_p(X)\to C_p(X)$ by $F=1+\partial W.$  The following is useful to simplify the calculation of the boundary map $\w{\partial}$:
\begin{lemma}{\rm\cite{FarleyHomology}}\label{lemma:SimplerBoundary}
The map $F^\infty$ is well defined and satisfies $\w{\partial}=\pi F^\infty\partial.$
\end{lemma}

\subsection{The Farley-Sabalka discrete gradient vector field}\label{subsec:FSDGVF}

Farley and Sabalka construct a discrete gradient vector field to study the unordered discrete configuration space $\UD{n}{\Gamma}.$  Before describing their construction, we first recall their ordering on the vertices of $\Gamma.$  To order the vertices, first choose a spanning tree $T$ of $\Gamma,$ choose an embedding of $T$ in the plane, and choose a vertex of degree one in $T$ to be labeled $\ast.$  Assign $\ast$ the number 0, then travel away from $\ast$ along $T$ and label the vertices of $\Gamma$ in order as they are first encountered.  Whenever a vertex which is essential in $T$ is encountered, travel along the leftmost edge first, then turn around whenever a vertex of degree one in $T$ is encountered.  Continuing in this manner assigns numbers to each vertex of $\Gamma.$  An example is given in Figure \ref{fig:VertexOrdering}.  We indicate the edges contained in the spanning tree $T$ with solid lines and edges of the graph $\Gamma$ which are not contained in the spanning tree with dashed lines.  We refer to the latter edges as \emph{deleted edges}.  Notice the graph $\Gamma$ is only sufficiently subdivided for $n=2.$

\begin{figure}[htb]
    \centering
    \includegraphics[height=6cm]{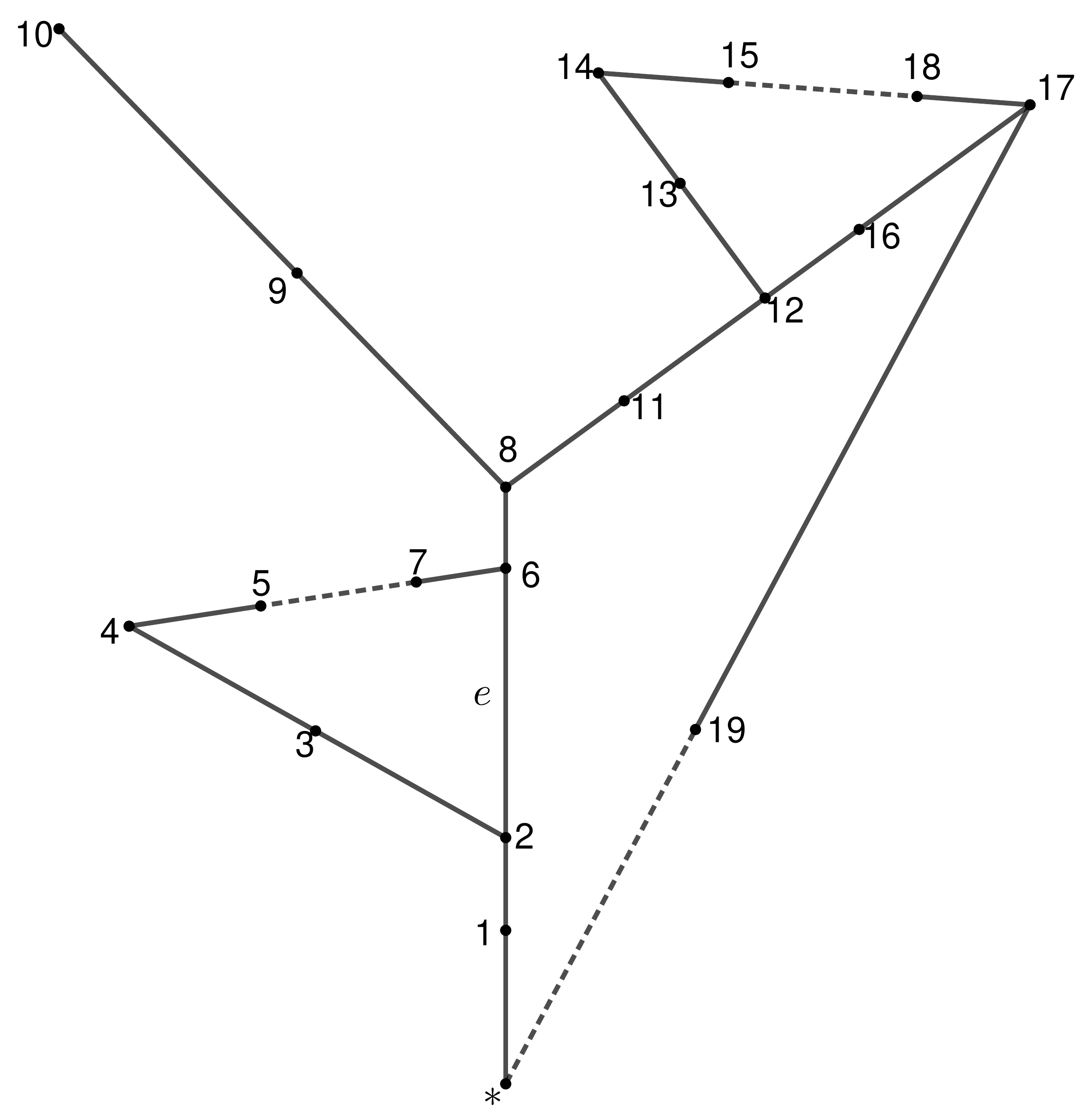}
    \caption{The ordering of the vertices of $\Gamma$}
    \label{fig:VertexOrdering}
\end{figure}

For an edge $e$ of $\Gamma,$ let $\iota(e)$ and $\tau(e)$ be the endpoints of $e,$ with $\tau(e)<\iota(e)$ in this ordering.  In Figure \ref{fig:VertexOrdering}, we have $\tau(e)=2$ and $\iota(e)=6.$  There is a notion of ``directions'' from each vertex $v\ne\ast$ of $\Gamma.$  These directions can be viewed as a numbering of the edges of $T$ incident to $v,$ clockwise from 0 to $\deg_T(v)-1,$ with 0 being the direction of the unique edge incident to $v$ which falls on the $T$-geodesic from $v$ to $\ast.$  For each vertex $u\ne\ast$ of $\Gamma,$ let $e^i(u)$ denote the edge in $T$ which is incident to $u$ in direction $i,$ let $v^i(u)=\iota(e^i(u))$ for $i\ne0,$ and let $v^0(u)=\tau(e^0(u)).$  In Figure \ref{fig:VertexOrdering}, the edge $e$ satisfies $e=e^2(2)=e^0(6),$ and the vertex 6 satisfies $6=v^2(2)=v^0(8).$

From here on, when referring to a cell, we mean the closure of that cell, so that we view cells of $\UD{n}{\Gamma}$ simply as collections of vertices and edges of $\Gamma.$  Consider a $k$-dimensional cell $c=\{e_1,\dots,e_k,v_1,\dots,v_{n-k}\}$ in $\UD{n}{\Gamma},$ where each $e_i$ is an edge of $\Gamma$ and each $v_i$ is a vertex of $\Gamma.$  For each vertex $v_i$ appearing in $c,$ if the edge $e^0(v_i)$ intersects some vertex $v_j\ne v_i$ or some edge $e_j$ appearing in $c,$ the vertex $v_i$ is said to be \emph{blocked} in $c.$  The vertex $\ast$ is also said to be blocked in any cell in which it appears.  Otherwise, $v_i$ is said to be \emph{unblocked} in $c.$  An edge $e_i$ in $c$ is said to be \emph{order-disrespecting in $c$} if either $e_i$ is an edge in $T$ and $c$ contains a vertex $v$ such that $v^0(v)=\tau(e_i)$ and $v<\iota(e_i),$ or if $e_i$ is a deleted edge.   In particular, if $e$ is an edge of $T$ which is order-disrespecting in some cell, then $\tau(e)$ must be an essential vertex.

We can now describe Farley and Sabalka's discrete gradient vector field on $\UD{n}{\Gamma}$ \cite{FarleySabalkaGraphBraidGroups}.  Recall this is a map $W$ which assigns to each $i$-dimensional cell either an $(i+1)$-cell or the element 0.  The construction of $W$ is given inductively by dimension.  Let $c$ be a 0-dimensional cell of $\UD{n}{\Gamma}.$  If every vertex of $c$ is blocked, let $W(c)=0.$  Otherwise, $c$ contains at least one unblocked vertex, and therefore it contains a minimal unblocked vertex $v$ (with respect to the ordering on the vertices).  In this case, let $W(c)$ be the one-dimensional cell obtained from $c$ be replacing $v$ with $e^0(v).$  Inductively, after $W$ is defined on the $(i-1)$-cells, define $W$ on the $i$-dimensional cells as follows.  For each $i$-cell $c,$ if each vertex of $c$ is blocked, or if $c=W(c')$ for some $(i-1)$-cell $c',$ let $W(c)=0.$  Otherwise, $c$ again contains a minimal unblocked vertex $v,$ and we let $W(c)$ be the $(i+1)$-cell obtained from $c$ by replacing $v$ with $e^0(v).$  Farley and Sabalka show that $W$ satisfies the properties of a discrete gradient vector field, which in turn classifies each cell as either redundant, collapsible, or critical. 

\begin{thm}{\rm\cite{FarleySabalkaGraphBraidGroups}}\label{thm:ClassificationOfCells}
Let $c$ be a cell of $\UD{n}{\Gamma}.$
\begin{enumerate}
\item The cell $c$ is redundant if and only if $c$ contains an unblocked vertex $v$ such that any order-respecting edge $e$ of $c$ satisfies $\iota(e)>v.$
\item The cell $c$ is collapsible if and only if $c$ contains some order respecting edge $e$ such that any vertex $v$ of $c$ satisfying $v<\iota(e)$ is blocked.
\item The cell $c$ is critical if and only if each vertex of $c$ is blocked and each edge of $c$ is order-disrespecting.
\end{enumerate}
\end{thm}
Using this description of the classification of cells together with Theorem \ref{thm:CriticalDimension}, Farley and Sabalka show the following:
\begin{thm}{\rm\cite{FarleySabalkaGraphBraidGroups}}\label{thm:UDnUpperBound}
Let $\beta=\beta_1(\Gamma)$ be the first Betti number of $\Gamma,$ and let $K=\min\left\{n,\left\lfloor\frac{n+\beta}{2}\right\rfloor,m(\Gamma)\right\}.$  A spanning tree can be chosen such that $\dim(c)\le K$ for any critical cell $c.$  In particular, $\UD{n}{\Gamma}$ is homotopy equivalent to a space of dimension $K.$
\end{thm}
This result gives general upper bounds on $\TC(\UD{n}{\Gamma})$: 

\begin{thm}\label{thm:TCUDnUB}
For any graph $\Gamma,$ we have $\TC(\UD{n}{\Gamma})\le2K+1,$ where $K=\min\left\{n,\left\lfloor\frac{n+\beta}{2}\right\rfloor,m(\Gamma)\right\}.$
\end{thm}
\begin{proof}
This follows immediately from Theorems \ref{thm:TCTools} and \ref{thm:UDnUpperBound}.
\end{proof}
\subsection{Homology and Cohomology}\label{subsec:Cohomology}
We now turn to a discussion of the homology groups and cohomology ring of $\UD{n}{\Gamma}.$  Our main goal in this subsection is to describe cohomological information in Lemmas \ref{lemma:Products} and \ref{lemma:DistinctHomologyClasses} which, together with Theorem \ref{thm:TCTools} will provide lower bounds on $\TC(\UD{n}{\Gamma})$ in the proofs of the main results of Section \ref{sec:MainResults}.

According to Theorem \ref{thm:MorseComplex}, we can compute the homology groups $H_*(\UD{n}{\Gamma})$ by studying the boundary operator $\w{\partial}=\pi F^\infty\partial.$  For a cell $c$ containing edges $e_1,\dots,e_k,$ we can without loss of generality assume that the edges satisfy $\iota(e_i)<\iota(e_{i+1})$ for each $i.$  The cellular boundary $\partial$ can be given as 
\[
\partial(c)=\sum_{i=1}^k(c_{\iota(e_i)}+c_{\tau(e_i)}),
\]
where $c_{\tau(e)}$ is the cell obtained from $c$ by replacing the edge $e$ with the endpoint $\tau(e),$ and $c_{\iota(e)}$ is defined similarly.  For the case in which $\Gamma$ is a tree, Farley and Sabalka show that the maps $\w{\partial}$ are as simple as possible.  
\begin{thm}{\rm\cite{FarleySabalka}}\label{thm:TreeBoundaries}
Let $\Gamma$ be a tree.  The boundary maps $\w{\partial}$ in $M_*(\UD{n}{\Gamma})$ are all trivial.  In particular, $H_i(\UD{n}{\Gamma})$ is isomorphic to a free $\Zm{2}$-module on the critical $i$-cells.
\end{thm}
In theory, it is possible to compute the boundaries $\w{\partial}$ in $M_*(\UD{n}{\Gamma})$ where $\Gamma$ is any graph, but in practice, this is often a difficult computation when approached directly.  However, Theorem \ref{thm:TreeBoundaries} can be used in the following simple situation.

\begin{cor}\label{cor:CellsInTrees}
If $c$ is a critical cell in $\UD{n}{\Gamma}$ which does not contain any deleted edges, then $\w{\partial}(c)=0.$
\end{cor}
\begin{proof}
This follows immediately from Theorem \ref{thm:TreeBoundaries} and the definition of $W.$
\end{proof}

Next, we describe the equivalence relation on cells given in \cite{FarleySabalkaCohomology}, which is used in the description of the cup product structure.

\begin{defn}{\rm\cite{FarleySabalkaCohomology}}\label{defn:EquivClassesTree}
Given two cells $c$ and $c'$ of $\UD{n}{\Gamma},$ where $\Gamma$ is an arbitrary graph, define $``\sim"$ by $c\sim c'$ if and only if $c$ and $c'$ share the same edges (so in particular $c$ and $c'$ are of the same dimension), and if $E$ is the union of edges in $c$ (and in $c'$), then for every connected component $C$ of $\Gamma-E,$ the number of vertices of $c$ in $C$ equals the number of vertices of $c'$ in $C.$
\end{defn}  
We will sometimes write $c\sim_\Gamma c'$ to emphasize the underlying graph.  Let $[c]$ (or $[c]_\Gamma$) denote the equivalence class of $c.$ We will denote the set of the edges in any cell of $[c]$ by $E([c]).$ Next, a partial order is defined on equivalence classes of cells.

\begin{defn}{\rm\cite{FarleySabalkaCohomology}}\label{defn:PartialOrderTree}
Given two equivalence classes $[c]$ and $[d]$ in $\UD{n}{\Gamma},$ write $[d]\le[c]$ if  there are representatives $c\in[c]$ and $d\in[d]$ such that $d$ is obtained from $c$ by removing some (possibly zero) edges of $c$ and replacing each of these edges with one of its endpoints.  
\end{defn}

For the case in which $\Gamma$ is a tree, Farley and Sabalka show the following:

\begin{lemma}{\rm\cite{FarleySabalkaCohomology}}\label{lemma:PropertiesOfEquivClassesandPOTrees}
Let $T$ be a tree.  The relation $``\le"$ is a well-defined partial order on the equivalence classes of cells in $\UD{n}{T},$ with the following properties:
\begin{enumerate}
\item If $c$ and $c'$ are $i$-cells of $\UD{n}{T}$ which contain the same edges, and $[c]$ and $[c']$ have a common upper bound, then $[c]=[c'].$
\item If $\{[c_1],\dots,[c_k]\}$ is a collection of distinct equivalence classes of 1-cells with a common upper bound, then the collection has a least upper bound, and if $e_i$ is the unique edge in $[c_i],$ then $e_i\cap e_j=\emptyset$ for $i\ne j.$
\item If $c$ is a $k$-cell in $\UD{n}{T},$ then there is a unique collection of equivalence classes of 1-cells $\{[c_1],\dots,[c_k]\}$ which has $[c]$ as an upper bound.
\item If $[c']\le [c]$ for some critical cell $c,$ then $[c']$ contains a critical cell.
\item If $c$ is critical, then $[c]$ contains only the cell $c$ and redundant cells.
\end{enumerate}
\end{lemma}
The proof of Lemma \ref{lemma:PropertiesOfEquivClassesandPOTrees} easily extends to give the following properties in the general case.
\begin{cor}\label{lemma:EquivPpties}
Let $\Gamma$ be any graph.  The relation $``\le"$ is a well-defined partial order on the equivalence classes of cells in $\UD{n}{\Gamma},$ with the following properties:
\begin{enumerate}
\item If $c$ and $c'$ are $i$-cells of $\UD{n}{\Gamma}$ which contain the same edges, and $[c]$ and $[c']$ have a common upper bound, then $[c]=[c'].$
\item If $\{[c_1],\dots,[c_k]\}$ is a collection of distinct equivalence classes of 1-cells with a common upper bound, and if $e_i$ is the unique edge in $[c_i],$ then $e_i\cap e_j=\emptyset$ for $i\ne j.$
\item If $c$ is a $k$-cell in $\UD{n}{\Gamma},$ then there is a unique collection of equivalence classes of 1-cells $\{[c_1],\dots,[c_k]\}$ which has $[c]$ as an upper bound.\label{ppty3}
\item If $[c']\le [c]$ for some critical cell $c,$ then $[c']$ contains a critical cell.
\end{enumerate}
\end{cor}

Note the weakening of property 2 and the omission of property 5.  Information regarding the product structure of $H^*(\UD{n}{\Gamma})$ is obtained via a map $\UD{n}{\Gamma}\to\widehat{\UD{n}{\Gamma}},$ where $\widehat{\UD{n}{\Gamma}}$ is a subcomplex of a high-dimensional torus.  Specifically, $\widehat{\UD{n}{\Gamma}}$ is a subcomplex of $\prod S^1_{[c]},$ where the product is taken over all equivalence classes of 1-cells in $\UD{n}{\Gamma}$ (the ordering of the product can be chosen arbitrarily), and for each equivalence class of 1-cells $[c],$ $S^1_{[c]}$ is a circle consisting of a single 0-cell and a single 1-cell.  A $k$-dimensional cell in this product corresponds to a collection $\mathcal{K}$ of $k$ distinct equivalence classes of 1-cells in $\UD{n}{\Gamma},$ and we label this $k$-cell by $\mathcal{K}.$  The complex $\widehat{\UD{n}{\Gamma}}$ is obtained from this product by removing open $k$-cells which are labeled by collections $\mathcal{K}$ which do not have upper bounds.  For each collection $\mathcal{K}$ of equivalence classes of 1-cells, let $\mathcal{K}^*$ denote the cellular cochain in $H^*(\widehat{\UD{n}{\Gamma}})$ dual to the cell labeled by $\mathcal{K}.$  If $\mathcal{K}=\{[c_1],\dots,[c_k]\},$ and $\mathcal{K}$ has an upper bound, then 
\[
\{[c_1]\}^\ast\smallsmile\cdots\smallsmile\{[c_k]\}^\ast=\mathcal{K}^\ast.
\]
If $\mathcal{K}$ does not have an upper bound, then 
\[
\{[c_1]\}^\ast\smallsmile\cdots\smallsmile\{[c_k]\}^\ast=0.
\]

Farley and Sabalka show that in the tree case, there is a map \[q\from \UD{n}{T}\to\widehat{\UD{n}{T}}\] which sends each $k$-dimensional cell $c$ homeomorphically to the cell labeled by $\mathcal{K}$ in $\widehat{\UD{n}{T}},$ where $\mathcal{K}$ is the unique collection of $k$ equivalence classes of 1-cells which has $[c]$ as its upper bound \cite{FarleySabalkaCohomology}.  The proof extends directly to the general graph case, where we replace $T$ with $\Gamma.$  The only exception is that in the tree case, the equivalence class of $k$-cells $[c]$ is the \textit{least} upper bound for the collection $\mathcal{K},$ so every cell $c$ in $\UD{n}{\Gamma}$ which is sent to $\mathcal{K}$ satisfies $c\in[c].$  In this case, the cell of $\widehat{\UD{n}{\Gamma}}$ labeled by $\mathcal{K}$ can be labeled by $[c].$  However, in general, it is possible that two (or more) distinct equivalence classes of $k$-cells $[c]$ and $[c']$ are both upper bounds for $\mathcal{K},$ so it is possible that $q(c)=q(c')=\mathcal{K}$ but $[c]\ne[c'].$  

For each cell $c$ of $\UD{n}{\Gamma},$ define a cellular cochain $\phi_{[c]}\in C^*(\UD{n}{\Gamma})$ by
\[
\phi_{[c]}(c')=\begin{cases}1,&\text{if }c'\sim c\\
0,&\text{otherwise} .
\end{cases}
\]
Each such cochain is a cocycle.  The following is a generalization of Proposition 4.5 in \cite{FarleySabalkaCohomology} and Proposition 3.3 in \cite{FarleyCohomology}.  

\begin{lemma}\label{lemma:Products}
Consider a collection $\mathcal{K}=\{[c_1],\dots,[c_k]\}$ of distinct equivalence classes of 1-cells.  We have 
\[
\phi_{[c_1]}\smallsmile\cdots\smallsmile\phi_{[c_k]}=\sum\phi_{[c'']},
\]
where the sum is taken over all equivalence classes of $k$-cells $[c'']$ which are upper bounds for $\mathcal{K}.$  In particular, if $\mathcal{K}$ has a least upper bound $[c],$ then 
\[
\phi_{[c_1]}\smallsmile\cdots\smallsmile\phi_{[c_k]}=\phi_{[c]},
\]
 and if $\mathcal{K}$ does not have an upper bound, then 
\[
\phi_{[c_1]}\smallsmile\cdots\smallsmile\phi_{[c_k]}=0.
\]
\end{lemma}
\begin{proof}
If $c'$ is any $k$-cell, then 
\begin{align*}
q^*(\mathcal{K}^*)(c')=\mathcal{K}^*(q(c'))&=\begin{cases}
1,&\text{if } q(c')=\mathcal{K}\\
0,&\text{otherwise}
\end{cases}\\
&=\begin{cases}
1,&\text{if $[c']$ is an upper bound for $\mathcal{K}$}\\
0,&\text{otherwise}
\end{cases}\\
&=\sum \phi_{[c'']}(c'),
\end{align*}
where the sum is taken over all equivalence classes of $k$-cells $c''$ such that $[c'']$ is an upper bound for $\mathcal{K}.$  If $[c]$ is the least upper bound for $\mathcal{K},$ then $q^*(\mathcal{K}^*)=\phi_{[c]}.$  In particular, since each $c_i$ is one-dimensional, the class $[c_i]$ is the least upper bound for the collection $\{[c_i]\},$ so $q^*(\{[c_i]\}^*)=\phi_{[c_i]}.$  Therefore, 
\[
\phi_{[c_1]}\smallsmile\cdots\smallsmile\phi_{[c_k]}=q^*(\{[c_1]\}^*\smallsmile\cdots\smallsmile\{[c_k]\}^*)=q^*(\mathcal{K}^*)=\sum \phi_{[c'']}.
\]
\end{proof}

\begin{lemma}\label{lemma:DistinctHomologyClasses}
Suppose $c$ and $c'$ are critical cells which represent linearly independent homology classes in $H_*^\mathcal{M}(\UD{n}{\Gamma})$ and have the property that the equivalence class $[c]$ (resp. $[c']$) only contains $c$ (resp. $c'$) and redundant cells.  Then $\phi_{[c]}$ and $\phi_{[c']}$ represent linearly independent cohomology classes in $H^*(\UD{n}{\Gamma}).$
\end{lemma}
\begin{proof}
Let $\mathcal{B}$ be a representative basis for $H_*^\mathcal{M}(\UD{n}{\Gamma})$.  Without loss of generality, we may assume that $c\in\mathcal{B}$ and $c'\in\mathcal{B},$ but neither $c$ nor $c'$ appears in any other linear combination in $\mathcal{B}.$  By Theorem \ref{thm:MorseComplex},  $\mathcal{B'}=\{f^\infty(b):b\in\mathcal{B}\}$ forms a basis for $H_*(\UD{n}{\Gamma}).$ 

By the universal coefficient theorem, we may identify $H^*(\UD{n}{\Gamma})$ with $\hom(H_*(\UD{n}{\Gamma}),\Zm{2}).$  For each $b\in\mathcal{B},$ let $b^*$ denote the dual of $f^\infty(b),$ so $\{b^*:b\in\mathcal{B}\}$ represents a basis for $H^*(\UD{n}{\Gamma}).$  Thus, the cocycles $c^*$ and $c'^*$ represent linearly independent classes in $H^*(\UD{n}{\Gamma}).$  The claim will then follow by showing $\phi_{[c]}=c^*$ and $\phi_{[c']}=c'^*$ in cohomology.  To show this, we consider evaluating $\phi_{[c]}$ and $\phi_{[c']}$ on the elements of the basis $\mathcal{B'}$ for $H_*(\UD{n}{\Gamma}).$  For any critical cell $\tilde{c},$ we have $f^\infty(\tilde{c})=\tilde{c}+\text{collapsible cells}$ (see Theorem \ref{thm:MorseComplex}).  But, by assumption, $[c]$ contains only the critical cell $c$ and redundant cells, so 
\[
\phi_{[c]}(f^\infty(\tilde{c}))=\phi_{[c]}(\tilde{c})+\phi_{[c]}(\text{collapsible cells})=\phi_{[c]}(\tilde{c})=
\begin{cases}
1,&\text{if }\tilde{c}=c\\
0,&\text{if }\tilde{c}\ne c.
\end{cases}
\]
Therefore, by the assumption about the basis, for each $f^\infty(b)$ in $\mathcal{B'}$ we have 
\[
\phi_{[c]}(f^\infty(b))=\begin{cases}
1,&\text{if }b=c\\
0,&\text{if }b\ne c,
\end{cases}
\]
so $\phi_{[c]}$ and $c^*$ represent the same cohomology class, as claimed.  An analogous argument shows $\phi_{[c']}$ and $c'^*$ are also cohomologous.
\end{proof}

\section{Main Results}\label{sec:MainResults}

Our approach to studying $\TC(\UD{n}{\Gamma})$ is as follows. Theorem \ref{thm:TCUDnUB} gives the upper bounds 
\[
\TC(\UD{n}{\Gamma})\le 2K+1,\quad\text{ where }\quad K=\min\left\{n,\left\lfloor\frac{n+\beta}{2}\right\rfloor,m(\Gamma)\right\}.
\]
To establish lower bounds, we will describe two critical $K$-cells $c$ and $c'$ which satisfy the conditions given in Lemma \ref{lemma:DistinctHomologyClasses} and have the property that $[c]$ and $[c']$ are the least upper bounds for the unique collections of $K$ equivalence classes of 1-cells which have $[c]$ and $[c']$ as their upper bounds.  Suppose $\{[c_1],\dots,[c_K]\}$ and $\{[c'_1],\dots,[c'_K]\}$ are these collections.  For each $c_i,$ let $\overline{\phi}_{[c_i]}$ denote the zero-divisor 
\[
\phi_{[c_i]}\otimes1+1\otimes\phi_{[c_i]}\in H^*(\UD{n}{\Gamma})\otimes H^*(\UD{n}{\Gamma}),
\]
and define $\overline{\phi}_{[c'_i]}$ analogously.  This gives rise to a product of $2K$ zero-divisors:
\[
\biggl(\prod_{i=1}^K\overline{\phi}_{[c_i]}\biggr)\cdot\biggl(\prod_{j=1}^K\overline{\phi}_{[c'_j]}\biggr)
=(\phi_{[c_1]}\cdots\phi_{[c_K]}\otimes \phi_{[c'_1]}\cdots\phi_{[c'_K]})+(\phi_{[c'_1]}\cdots\phi_{[c'_K]}\otimes \phi_{[c_1]}\cdots\phi_{[c_K]})
+\text{ Other Terms.}
\]

By Lemma \ref{lemma:Products}, we have $\phi_{[c_1]}\cdots\phi_{[c_K]}=\phi_{[c]}$ and $\phi_{[c'_1]}\cdots\phi_{[c'_K]}=\phi_{[c']},$ so we can write 

\[
\biggl(\prod_{i=1}^K\overline{\phi}_{[c_i]}\biggr)\cdot\biggl(\prod_{j=1}^K\overline{\phi}_{[c'_j]}\biggr)
=\phi_{[c]}\otimes\phi_{[c']}+\phi_{[c']}\otimes\phi_{[c]}+\text{ Other Terms}.
\]

By Lemma \ref{lemma:DistinctHomologyClasses}, $\phi_{[c]}\otimes\phi_{[c']}+\phi_{[c']}\otimes\phi_{[c]}$ is non-zero.  If neither $\phi_{[c]}$ nor $\phi_{[c']}$ is cohomologous to any term in Other Terms, it will follow that the entire product is non-zero, establishing $\TC(\UD{n}{\Gamma})\ge2K+1.$

Before proving Theorem \ref{thm:S_Graph}, we first define the class of $S$-graphs.  Fix a graph $\Gamma$ and a vertex $v$ in $\Gamma.$  A cycle in $\Gamma$ is called a \emph{simple cycle} if it contains exactly one essential vertex.  A component $C$ of $\Gamma-\{v\}$ is said to be a \emph{simple component} if each cycle in $C\cup\{v\}$ is a simple cycle.

\begin{defn}\label{defn:S_Graph}
A graph $\Gamma$ is said to be an $S$-graph if it has the following properties:
    \begin{enumerate}        
    \item $\Gamma$ contains a spanning tree $T$ such that each deleted edge $\w{e}$ has an endpoint at an essential vertex and each essential vertex $v\in\Gamma$ satisfies $\deg_T(v)\ge 4.$
    \item For each essential vertex $v,$ at least one component $C$ of $\Gamma-\{v\}$ is a simple component, and the edge of $T$ contained in $C\cup\{v\}$ which meets $v$ does not fall in direction 0 from $v.$
    \end{enumerate}
\end{defn}
For example, the graph in Figure \ref{fig:S_Graph} is an $S$-graph.

\begin{figure}[htb]
    \centering
    \includegraphics[height=4cm]{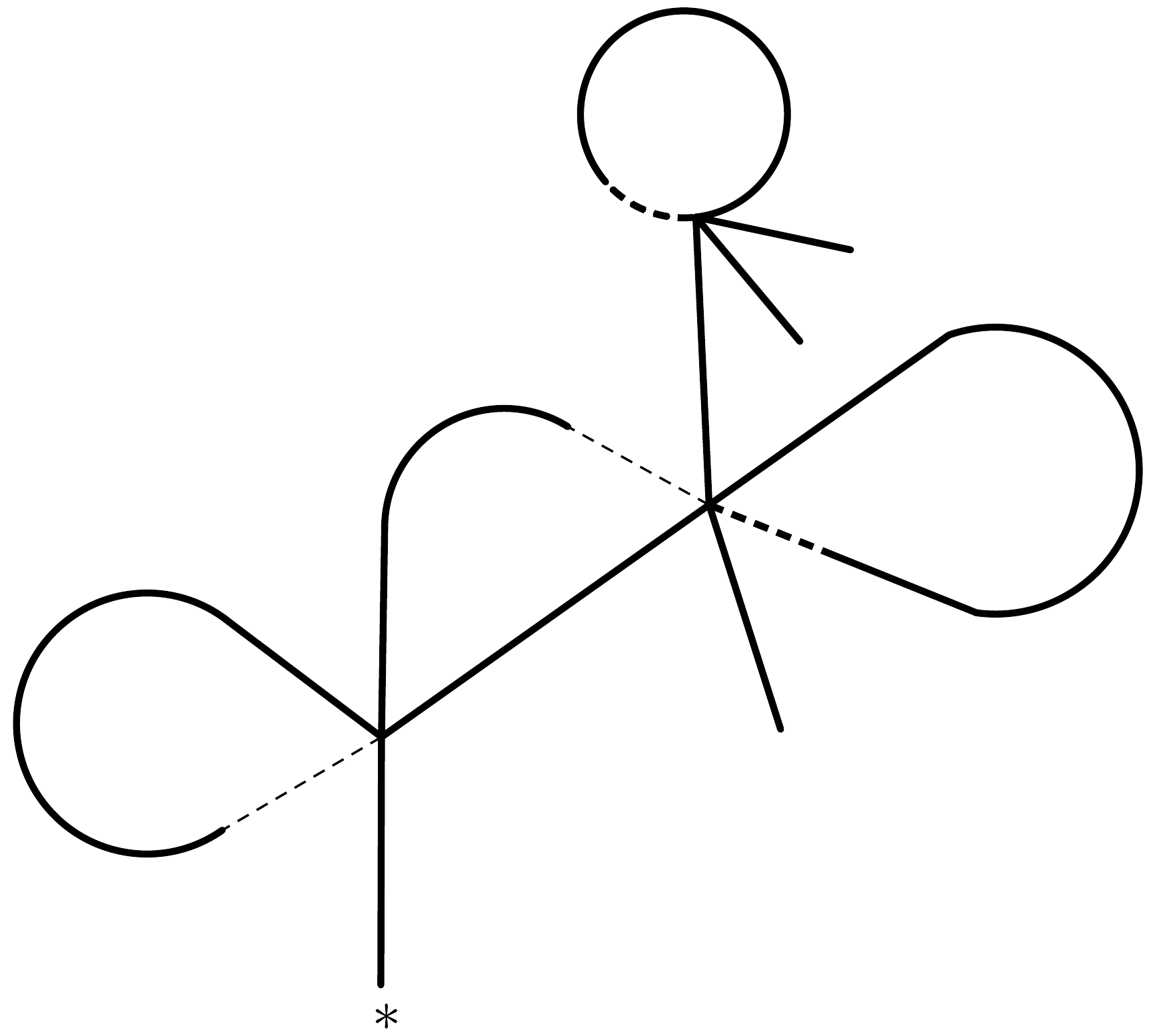}
    \caption{An $S$-graph $\Gamma$ and its spanning tree $T$}
    \label{fig:S_Graph}
\end{figure}

 Note the edge in property 2 is unique, for if there were two edges $e,e'$ of $T$ contained in $C\cup\{v\}$ which meet $v,$ then $C\cup\{v\}$ would necessarily contain a cycle on which both $e$ and $e'$ fall.  Such a cycle must contain a deleted edge $\w{e}$ which does not meet $v,$ and by assumption, $\w{e}$ has an endpoint at an essential vertex.  This contradicts the assumption that $C$ is a simple component.  So the edge must be unique, and we may refer to the direction of this edge as the ``direction of $C$."  If a simple component $C$ of $\Gamma-\{v\}$ does not fall in direction 0, then it is clear that an embedding can be chosen so that it falls in any desired non-zero direction from $v.$

Also, property 2 implies that every $S$-graph is fully articulated (see Theorem \ref{thm:TCFullyArticulated}).  However, the converse is not true.  For example, if $\Gamma$ is a wedge of 2 or 3 circles, then $\Gamma$ is fully articulated, but it is impossible to choose a spanning tree as in property 1, so $\Gamma$ is not an $S$-graph.  There are also graphs which admit spanning trees as in property 1 but do not satisfy property 2.  However, the class of $S$-graphs is fairly large.  More specifically, every graph is a subgraph of some $S$-graph.  To see this, note that every graph has a spanning tree with the property that each deleted edge has an endpoint at an essential vertex (see \cite{FarleySabalkaGraphBraidGroups} for a proof).  For any graph $\Gamma$ with such a spanning tree, if for each essential vertex $v,$ we add three edges at $v,$ the resulting graph is an $S$-graph.  Finally, note that with this choice of spanning tree, any order disrespecting edge must have an endpoint at an essential vertex, so it follows that $\dim(c)\le m(\Gamma)$ for any critical cell $c.$

The approach to proving Theorem \ref{thm:S_Graph} is the approach described at the beginning of this section.  In Definition \ref{defn:Phi} we construct the cells which satisfy the conditions of Lemma \ref{lemma:DistinctHomologyClasses} and verify these conditions in Lemmas \ref{lemma:PhiPpty1} and \ref{lemma:PhiPpty2}.  We will first need a preliminary result.

Consider any graph $\Gamma$ with spanning tree $T$ and a cell $c$ of $\UD{n}{\Gamma}$ which does not contain any deleted edges. The cell $c$ can be viewed as a cell in $\UD{n}{T}$ and we have 
\[
[c']_T\le[c]_T\Rightarrow[c']_\Gamma\le[c]_\Gamma.
\]
We have the following converse for certain types of graphs and cells.

\begin{lemma}\label{lemma:simplecycles}
Let $\Gamma$ be a graph in which each cycle is simple, let $T\subset \Gamma$ be a spanning tree such that each deleted edge has an endpoint at an essential vertex, and let $c$ and $c'$ be cells in $\UD{n}{\Gamma}$ such that each edge in each cell is an edge of $T$ with an endpoint at an essential vertex.  Then
\[
[c']_\Gamma\le[c]_\Gamma\Rightarrow[c']_T\le[c]_T.
\]
\end{lemma}
\begin{proof}
We first introduce some notation.  For a cell $c$ in $\UD{n}{\Gamma}$ which does not include any deleted edges and a connected component $C\subset \Gamma-E(c),$ let $V_\Gamma(c,C)$ denote the number of vertices of $c$ in $C.$  Define $V_{T}(c,C)$ analogously for subsets $C\subset T-E(c).$

First suppose $[c']_\Gamma=[c]_\Gamma,$ so $E(c)=E(c'),$ and $V_\Gamma(c,C)=V_{\Gamma}(c',C)$ for each component $C$ of $\Gamma-E(c).$  Fix such a component $C,$ and let $C_T\subset C$ denote the subspace of $T$ obtained from $C$ by removing the interiors of all deleted edges.  In general, $C_T$ is a union of components $C_1,\dots,C_l$ of $T-E(c)$ and $V_\Gamma(c,C)=V_T(c,C_1)+\cdots+V_T(c,C_l).$  We claim that under the hypotheses, $C_T$ is itself a connected component of $T-E(c).$

For the sake of contradiction, suppose $C$ is connected but $C_T$ is disconnected.  Let $C'_T$ and $C''_T$ be two connected components of $C_T.$  In $\Gamma-E(c),$ the subspaces $C'_T$ and $C''_T$ both fall in the connected component $C,$ so $C$ must contain at least one deleted edge $\w{e}$ with endpoints $x$ and $y$ in $C'_T$ and $C''_T,$ respectively.  By assumption, one of these endpoints (say $x$) must be essential.  Furthermore, since $T$ is connected, there is a path in $T$ connecting $x$ and $y$, which, together with $\w{e},$ forms a cycle $\cycle.$  Since $C'_T$ and $C''_T$ are disconnected in $T-E(c),$ it must be the case that $\cycle$ intersects the essential endpoint $v$ of some edge $e\in E(c).$  Since $x\in C'_T$ and $v\notin C'_T,$ we have $v\ne x.$  But, $v$ and $x$ are both essential, contradicting the assumption that $\cycle$ is a simple cycle. So, $C_T$ is connected, and it follows that $V_\Gamma(c,C)=V_T(c,C_T),$ and similarly, $V_\Gamma(c',C)=V_T(c',C_T).$  The choice of $C$ was arbitrary, so for all components $C$ of $\Gamma-E(c),$ the subspace $C_T$ is is connected, and we have $V_T(c,C_T)=V_\Gamma(c,C)=V_\Gamma(c',C)=V_T(c',C_T).$  Therefore $[c']_T=[c]_T.$

Now, if $[c']_\Gamma<[c]_\Gamma,$ there are cells $\hat{c}'\in [c']_\Gamma$ and $\hat{c}\in [c]_\Gamma$ such that $\hat{c}'$ is obtained from $\hat{c}$ by removing certain edges and replacing each with one of its endpoints.  But, by the argument above, we have $\hat{c}'\in [c']_T$ and $\hat{c}\in [c]_T,$ so $[c']_T<[c]_T.$
\end{proof}

Fix an $S$-graph $\Gamma$ with a spanning tree $T$ as in Definition \ref{defn:S_Graph}.  We now define the cells that satisfy the conditions of Lemma \ref{lemma:DistinctHomologyClasses}.
\begin{defn}\label{defn:Phi}
Let $\Gamma$ be an $S$-graph, and without loss of generality, assume that for each essential vertex $v,$ the edge in direction 1 from $v$ falls in a simple component of $\Gamma-\{v\}$.  Let $n$ be an integer satisfying $n\ge 2m(\Gamma),$ and let $v_1,\dots,v_m$ be the essential vertices of $\Gamma.$  For each $S=(s_1,\dots,s_m)\in \{2,3\}^m,$ define the critical $m$-cell $\Phi_S(\Gamma)$ in $\UD{n}{\Gamma}$ by 
\[
\Phi_S(\Gamma)=\{e^{s_1}(v_1),\dots,e^{s_m}(v_m),v^1(v_1),\dots,v^1(v_m),w_1,\dots,w_{n-2m}\}
\]
where $w_1,\dots,w_{n-2m}$ are blocked vertices in the component of $\Gamma-{v_1}$ in direction 1 from $v_1$ (if $n>2m).$
\end{defn}

Figure \ref{fig:S_GraphCells} shows the cells $\Phi_S$ and $\Phi_{S'}$ for $S=(2,\dots,2)$ and $S'=(3,\dots,3),$ where $n=8$ and the graph $\Gamma$ is as in Figure \ref{fig:S_Graph}.  Here, we depict a cell by indicating which vertices and edges are to be included in that cell.  Since we assume the graph $\Gamma$ is sufficiently subdivided, there are many vertices of degree 2 in $\Gamma;$ we make no indication of these vertices in the figure.
\begin{figure}[ht]
    \centering
    \includegraphics[height=4cm]{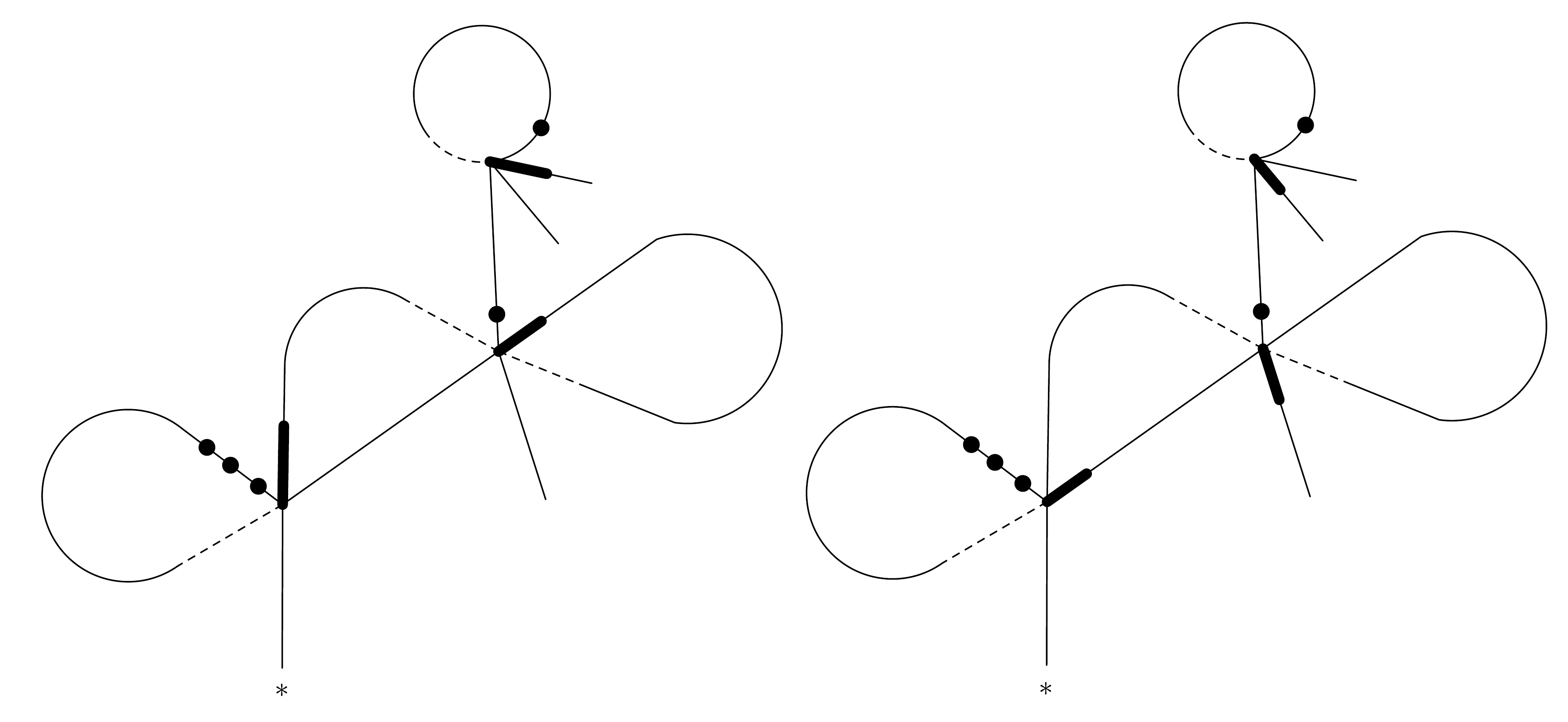}
    \caption{The cells $\Phi_{2,\dots,2}$ and $\Phi_{3,\dots,3}$ in $\UD{8}{\Gamma},$ where $\Gamma$ is as in Figure \ref{fig:S_Graph}.}
    \label{fig:S_GraphCells}
\end{figure}
\begin{lemma}\label{lemma:PhiPpty1}
For any distinct subsets $S,S'\in\{2,3\}^m,$ the cells $\Phi_S(\Gamma)$ and $\Phi_{S'}(\Gamma)$ represent distinct homology classes in $H_*^\mathcal{M}(\UD{n}{\Gamma}).$
\end{lemma}
\begin{proof}
Each cell of the form $\Phi_S(\Gamma)$ is a cell in $\UD{n}{T},$ so $\widetilde{\partial}\left(\Phi_S(\Gamma)\right)=0$ by Corollary \ref{cor:CellsInTrees}, so each $\Phi_S(\Gamma)$ is a cycle in the Morse complex.  As mentioned in the remarks following Definition \ref{defn:S_Graph}, with the choice of spanning tree, there are no critical cells of dimension greater than $m(\Gamma).$  In particular, $M_{m(\Gamma)+1}(\UD{n}{\Gamma})=0,$ showing each $m(\Gamma)$-dimensional cycle represents a distinct homology class.
\end{proof}
\begin{lemma}\label{lemma:PhiPpty2}
The equivalence class $[\Phi_S(\Gamma)]$ contains only the cell $\Phi_S(\Gamma)$ and redundant cells.  Furthermore, $[\Phi_S(\Gamma)]$ is the least upper bound for the unique collection of $m$ equivalence classes of 1-cells which has $[\Phi_S(\Gamma)]$ as its upper bound.
\end{lemma}
\begin{proof}
To prove the first statement, let $\Phi=\Phi_S(\Gamma),$ and suppose $c\sim \Phi.$  That is, $E(c)=E(\Phi)$ and the number of vertices of $c$ in each component $C$ of $\Gamma-E(c)$ equals the number of vertices of $\Phi$ in $C.$  In other words, we can obtain $c$ from $\Phi$ by replacing each vertex $v$ of $\Phi$ with other another vertex $v'$ (which may equal $v$) in the component of $\Gamma-E(\Phi)$ which contains $v.$  In particular, $c$ can be viewed as a cell in $\UD{n}{T}.$  The cell $\Phi$ contains an edge at each essential vertex, and since each deleted edge has an endpoint at an essential vertex, each such $v'$ must fall in the same component of $T-E(\Phi)$ as does $v.$  This shows that $c\sim_\Gamma \Phi$ if and only if $c\sim_{T}\Phi.$  Therefore, by Lemma \ref{lemma:PropertiesOfEquivClassesandPOTrees}, either $c=\Phi$ or $c$ is redundant, when viewed as a cell in $\UD{n}{T},$ but the classification of a cell in $\UD{n}{T}$ as being critical, collapsible, or redundant remains unchanged when viewing it as a cell in $\UD{n}{\Gamma}$ .  This proves the first statement.

For the second statement, suppose $\Phi'$ is an $m$-dimensional cell such that $[\Phi']$ is an upper bound for the unique collection $\mathcal{K}=\{[c_1],\dots,[c_m]\}$ of $m$ equivalence classes of 1-cells which has $[\Phi]$ as an upper bound.  To show $[\Phi]$ is the least upper bound for $\mathcal{K},$ we must show $[\Phi]=[\Phi'].$  Without loss of generality, assume the unique edge $e_i$ in $c_i$ satisfies $\tau(e_i)=v_i.$  The fact that $[\Phi']$ is an upper bound for $\mathcal{K}$ implies $E(\Phi')=E(\Phi)=\{e_1,\dots,e_m\}.$  For a component $C$ of $\Gamma-E(\Phi),$ let $V(\Phi,C)$ denote the number of vertices of $\Phi$ in the component $C,$ and define $V(\Phi',C)$ analogously.  Since the total number of vertices and edges in any cell of $\UD{n}{\Gamma}$ must equal $n,$ and there are a total of $m$ edges in each cell, we have 
\[
\sum_{C}V(c',C)=\sum_{C}V(\Phi',C)=n-m,
\]
where both sums are taken over the collection of all connected components $C\subset \Gamma-E(\Phi).$  We will show $[\Phi]=[\Phi']$ by showing $V(\Phi,C)=V(\Phi',C)$ for each such $C.$

For each $i,$ let $\hat{c}_i$ denote the cell obtained from $\Phi$ by replacing each edge $e_j,\ j\ne i,$ with the endpoint $\iota(e_j),$ and let $\hat{c}'_i$ denote the cell obtained from $\Phi'$ by the analogous procedure.  Since $[\Phi]$ and $[\Phi']$ are upper bounds for $\mathcal{K},$ we have $\hat{c}_i,\hat{c}'_i\in[c_i].$

Fix an essential vertex $v_i$ and let $\Gamma_i$ be the simple component of $\Gamma-\{v_i\}$ in direction 1 from $v_i.$  For a cell $d$ in $\UD{n}{\Gamma},$ let $d_{\Gamma_i}$ denote the cell consisting of the vertices and edges of $d$ which are contained in $\Gamma_i.$  Each such $d_{\Gamma_i}$ is a cell in $\UD{n'}{\Gamma_i}$ for some $n'\le n.$  Since $[\Phi]$ and $[\Phi']$ are both upper bounds for $[c_i],$ it must be the case that $\Phi_{\Gamma_i}$ and $\Phi'_{\Gamma_i}$ are cells in $\UD{n'}{\Gamma_i}$ for the same $n'.$  Furthermore, for any essential vertex $v_j$ in $\Gamma_i,$ any component of $\Gamma_i-e_j,$ except the component in direction 0 from $v_j,$ is a component of $\Gamma-e_j.$

To simplify notation, we will write $[d_{\Gamma_i}]_{\Gamma_i}$ as $[d]_{\Gamma_i}.$  It is clear that $[\Phi]_{\Gamma_i}$ and $[\Phi']_{\Gamma_i}$ are both upper bounds for the collections
\[
\{[\hat{c}_{j_1}]_{\Gamma_i},\dots,[\hat{c}_{j_{m'}}]_{\Gamma_i}\}\quad\text{and}\quad\{[\hat{c}'_{j_1}]_{\Gamma_i},\dots,[\hat{c}'_{j_{m'}}]_{\Gamma_i}\},
\]
respectively, where $v_{j_1},\dots,v_{j_{m'}}$ are the essential vertices in $\Gamma_i.$

Fix a $j\in\{j_1,\dots,j_{m'}\}.$  We claim that $[\hat{c}_j]_{\Gamma_i}=[\hat{c}'_j]_{\Gamma_i}.$ For the sake of contradiction, assume this is not the case.  Then, there must be a component $C$ of $\Gamma_i-e_j$ such that the number of vertices of $(\hat{c}_j)_{\Gamma_i}$ in $C$ differs from the number of vertices of $(\hat{c}'_j)_{\Gamma_i}$ in $C.$  Furthermore, since the total number of vertices of $(\hat{c}_j)_{\Gamma_i}$ and $(\hat{c}'_j)_{\Gamma_i}$ must equal $n'-1,$ there must be at least two such components, so we may assume $C$ does not fall in direction 0 from $v_j.$  Then, $C$ is a component of $\Gamma-e_i,$ so the number of vertices of $\hat{c}_j$ in $C$ must equal the number of vertices of $\hat{c}'_j$ in $C,$ since $[\hat{c}_j]_\Gamma=[\hat{c}'_j]_\Gamma.$  But, the number of vertices of $\hat{c}_j$ in $C$ equals the number of vertices of $(\hat{c}_j)_{\Gamma_i}$ in $C$ and the number of vertices of $\hat{c}'_j$ in $C$ equals the number of vertices of $(\hat{c}'_j)_{\Gamma_i}$ in $C,$ arriving at a contradiction.  Therefore
\[
\{[\hat{c}_{j_1}]_{\Gamma_i},\dots,[\hat{c}_{j_{m'}}]_{\Gamma_i}\}=\{[\hat{c}'_{j_1}]_{\Gamma_i},\dots,[\hat{c}'_{j_{m'}}]_{\Gamma_i}\},
\]
and $[\Phi]_{\Gamma_i}$ and $[\Phi']_{\Gamma_i}$ are both upper bounds for this collection.

Let $T_i=T\cap\Gamma_i.$  Since $\Gamma_i$ is a simple component, $T_i$ is a spanning tree of $\Gamma_i.$  Lemma \ref{lemma:simplecycles} shows that 
$[\Phi]_{T_i}$ and $[\Phi']_{T_i}$ are both upper bounds for the collection $
\{[\hat{c}_{j_1}]_{T_i},\dots,[\hat{c}_{j_{m'}}]_{T_i}\},$ and therefore by Lemma \ref{lemma:PropertiesOfEquivClassesandPOTrees}, $[\Phi]_{T_i}=[\Phi']_{T_i},$ so $[\Phi]_{\Gamma_i}=[\Phi']_{\Gamma_i}.$  In particular, if $C_i$ is the component of $\Gamma_i-E(c_{\Gamma_i})$ in direction 1 from $v_i,$ then 
\[
V(\Phi,C_i)=V(\Phi_{\Gamma_i},C_i)=V(\Phi'_{\Gamma_i},C_i)=V(\Phi',C_i).
\]
Now, each vertex $v$ of $\Phi$ falls in exactly one such $C_i,$ so
\[
n-m=\sum_{i=1}^mV(\Phi,C_i)=\sum_{i=1}^mV(\Phi',C_i),
\]
so $V(\Phi',C)=0$ for any component $C$ which is not of the form $C=C_i$ for some $i.$  Thus, $V(\Phi,C)=V(\Phi',C)$ for every component of $\Gamma-E(\Phi),$ so $[\Phi]=[\Phi'],$ completing the proof.
\end{proof}

\begin{proof}[Proof of Theorem \ref{thm:S_Graph}]
The upper bound $\TC(\UD{n}{\Gamma})\le2m(\Gamma)+1$ is given in Theorem \ref{thm:TCUDnUB}.  To establish the lower bound, we will use the cohomological lower bounds in Theorem \ref{thm:TCTools}.  Let $S_2=(2,2,\dots,2)\in\{2,3\}^m,$ and let $S_3=(3,3,\dots,3)\in\{2,3\}^m;$ let $\Phi_2=\Phi_{S_2}(\Gamma)$ and let $\Phi_3=\Phi_{S_3}(\Gamma).$  Let $\{[c_1],\dots,[c_m]\}$ and $\{[d_1],\dots,[d_m]\}$ be the unique collections of equivalence classes of 1-cells which have $[\Phi_2]$ and $[\Phi_3]$ as their upper  bounds.  Lemma \ref{lemma:PhiPpty2} shows $[\Phi_2]$ and $[\Phi_3]$ are the least upper bounds of these collections.  Without loss of generality, assume the unique edges in $[c_i]$ and $[d_i]$ have an endpoint at $v_i.$  Consider the product of zero divisors 
\begin{equation}
\biggl(\prod_{i=1}^m\overline{\phi}_{[c_i]}\biggr)\cdot\biggl(\prod_{j=1}^m\overline{\phi}_{[d_j]}\biggr)
=\phi_{[\Phi_2]}\otimes\phi_{[\Phi_3]}+\phi_{[\Phi_3]}\otimes\phi_{[\Phi_2]}+\text{Other Terms.}\label{eqn:ZD}
\end{equation}
Any nonzero term in ``Other Terms" is of the form $\alpha\otimes\beta,$ where $\alpha$ and $\beta$ are of the form 
\begin{equation}
\phi_{[c_{i_1}]}\cdots\phi_{[c_{i_s}]}\phi_{[c_{j_1}]}\cdots\phi_{[d_{j_{m-s}}]},\quad0<s<m.
\label{eqn:OT}
\end{equation}
Let $I=\{i_1,\dots,i_s\}$ and let $J=\{j_1,\dots,j_{m-s}\}.$  If $I$ and $J$ have a common element $k,$ then both $[c_k]$ and $[d_k]$ contain an edge $e$ satisfying $\tau(e)=v_k,$ so by Lemma \ref{lemma:EquivPpties} the collection $\{[c_{i_1}],\dots,[c_{i_s}],[d_{j_1}],\dots,[d_{j_{m-s}}]\}$ does not have an upper bound, and then by Lemma \ref{lemma:Products}, the product in (\ref{eqn:OT}) is zero.  Otherwise, $I\cup J=\{1,\dots,m\},$ and if $S=(\epsilon_1,\dots,\epsilon_m)\in\{2,3\}^m$ where $\epsilon_i=2$ if $i\in I$ and $\epsilon_i=3$ if $i\in J,$ then $[\Phi_S(\Gamma)]$ is an upper bound for the collection $\{[c_{i_1}],\dots,[c_{i_s}],[d_{j_1}],\dots,[d_{j_{m-s}}]\}.$  Again, Lemma \ref{lemma:PhiPpty2} shows $[\Phi_S(\Gamma)]$ is the least upper bound for this collection, so in this case, the product in (\ref{eqn:OT}) equals $\phi_{[\Phi_S(\Gamma)]}.$

So, the nonzero terms in ``Other Terms" are of the form $\phi_{[\Phi_{S'}(\Gamma)]}\otimes\phi_{[\Phi_{S''}(\Gamma)]},$ where $S'$ and $S''$ are sequences in $\{2,3\}^m$ which are neither $(2,2,\dots,2)$ nor $(3,3,\dots,3).$  Since all cells of the form $\Phi_S(\Gamma)$ represent distinct homology classes, the product in (\ref{eqn:ZD}) is nonzero, and the lower bounds in Theorem \ref{thm:TCTools} show $\TC(\UD{n}{\Gamma})\ge2m+1.$ 
\end{proof}

We now move to proving Theorem \ref{thm:VertexDisjointUnordered}.  The proof will involve the Morse boundary of cells containing deleted edges.  The use of a spanning tree $T\subset\Gamma$ in which each deleted edge has endpoints of degree 2 in $\Gamma$ will help simplify this computation.  We refer to such spanning trees as \emph{essential spanning trees}. It is clear that any graph can be subdivided so that it has an essential spanning tree.  The main benefit of using essential spanning trees is the following observation.  
\begin{lemma}\label{lemma:DeletedEdgeBlocks}
If $T$ is an essential spanning tree of $\Gamma$ and $\w{e}$ is a deleted edge with $\tau(\w{e})\ne\ast,$ then in any cell $c$ containing $\w{e},$ no vertex in $c$ can be blocked by $\w{e}.$ If $\tau(\w{e})=\ast,$ then the only vertex which can be blocked by $\w{e}$ is the vertex with label 1.
\end{lemma}

\begin{proof}
Suppose $c$ is a cell containing some deleted edge $\w{e}$ and a vertex $v$ blocked by $\w{e}.$ Then, $v$ is blocked by $v'$ in $c_{v'},$ where $v'$ either $\iota(\w{e})$ or $\tau(\w{e}).$ 

If either $\tau(\w{e})\ne\ast,$ or $\tau(\w{e})=\ast$ and $v'=\iota(\w{e}),$ then the edges $\w{e},\ e^0(v),$ and $e^0(v')$ are three edges which have an endpoint at $v',$ so $v'$ has degree at least three, contradicting the assumption that $T$ is an essential spanning tree. 

If $\tau(\w{e})=\ast$ and $v'=\tau(\w{e}),$ and $v$ is not the vertex labeled 1, then the edges $\w{e},\ e^0(v),$ and $e^0(1)$ are three edges which have an endpoint at $\ast,$ again contradicting the assumption that $T$ is essential.
\end{proof}

To prove that the boundary operators in the tree case are all zero (Theorem \ref{thm:TreeBoundaries}), Farley introduces the following ``reduction."
\begin{defn}{\rm\cite{FarleyHomology}}
If a cell $c$ of $\UD{n}{\Gamma}$ is redundant, let $v$ be the minimal unblocked vertex of $c$, and let $r(c)$ be the cell obtained from $c$ by replacing $v$ with $v'=v^0(v).$
\end{defn}
In other words, for a redundant $k$-cell $c,$ the $k$-cell $r(c)$ is obtained by moving the minimal unblocked vertex of $c$ one step closer to $\ast.$

\begin{lemma}{\rm\cite{FarleyHomology}}\label{lemma:Rinfty}
If $T$ is a tree, and $c$ is a redundant cell in $\UD{n}{T},$ we have $\pi F^\infty(c)=\pi F^\infty r(c).$
\end{lemma}

In Lemma \ref{lemma:rIinfty}, we give an analogous version of Lemma \ref{lemma:Rinfty} for certain types of cells in the configuration spaces for general graphs, where the reduction $r(c)$ is replaced with the ``initial reduction" $r_I(c)$ defined below.  Here, for vertices $x$ and $y,$ the interval $[x,y]$ denotes all vertices $z$ satisfying $x\le z\le y$ in the ordering of the vertices given by the spanning tree.

\begin{defn}\label{defn:InitialReduction}
Given a redundant cell $c$ of $\UD{n}{\Gamma},$ let $v$ be the minimal unblocked vertex of $c.$  If $v'=v^0(v),$ and for all deleted edges $\w{e}\in c,$ either 
\begin{enumerate}
\setlength\itemsep{0mm}
\item $[v',v]\subset [\tau(\w{e}),\iota(\w{e})],$ or 
\item $[v',v]\cap [\tau(\w{e}),\iota(\w{e})]=\emptyset,$ 
\end{enumerate}
define the initial reduction $r_I(c)$ to be the cell obtained by replacing $v$ with $v'.$  Otherwise, let $r_I(c)=c.$ In the latter case, we say $c$ is defective, and in the former, we say $v$ is non-defective.
\end{defn}
 Note that if a redundant cell $c$ is non-defective, then $r_I(c)=r(c).$  The following is stated in \cite{FarleyHomology} for the tree case; the proof generalizes to the general graph case, making use of essential spanning trees.  
\begin{lemma}[See Lemma 3.5(1) in \cite{FarleyHomology}]\label{lemma:BlockedInInterval}
Let $\Gamma$ be any graph with an essential spanning tree $T.$  If a cell $c$ of $\UD{n}{\Gamma}$ contains an order-respecting edge $e$ such that any vertex $v$ of $c$ which falls in $(\tau(e),\iota(e))$ is blocked, then $\pi F^\infty(c)=0.$
\end{lemma}

\begin{cor}\label{cor:DeletedEdgeEndpoints}
Let $\Gamma$ be any graph with essential spanning tree $T$.  If $c$ is a non-defective redundant cell, and $\w{e}$ is a deleted edge in $c,$ then 
\[\pi F^\infty\left((Wc)_{\tau(\w{e})}\right)=\pi F^\infty\left((Wc)_{\iota(\w{e})}\right)=0.
\]
\end{cor}

\begin{proof}
Let $v$ be the minimal unblocked vertex of $c,$ and let $e'=e^0(v),$ so that $Wc$ is obtained from $c$ by replacing $v$ with $e'.$ It is straightforward to check that the edge $e'$ is order-respecting in $Wc$.  Since $c$ is non-defective, it has the property that if $v'=\tau(e'),$ then for every deleted edge $\w{e}$ in $c$, we have either $[v',v]\subset [\tau(\w{e}),\iota(\w{e})],$ or $[v',v]\cap [\tau(\w{e}),\iota(\w{e})]=\emptyset,$ so we have the following possibilities for each deleted edge $\w{e}:$
\begin{enumerate}
\item $v'<v<\tau(\w{e})<\iota(\w{e})$
\item $\tau(\w{e})<\iota(\w{e})<v'<v$
\item $\tau(\w{e})<v'<v<\iota(\w{e})$
\end{enumerate}

Since $v$ is the minimal unblocked vertex, there are no unblocked vertices of $c$ in $(v',v),$ and therefore there are no unblocked vertices of $Wc$ in $(v',v).$  Since $T$ is an essential spanning tree, the only blocked vertices in $Wc$ which may become unblocked in $(Wc)_{\tau(\w{e})}$ (resp. $(Wc)_{\iota(\w{e})}$) are $\tau(\w{e})$ (resp. $\iota(\w{e})$ or the vertex labeled 1).  But, in all three cases above, we see that neither $\tau(\w{e})$ nor $\iota(\w{e})$ is in $(v',v).$  Furthermore, it is impossible that the vertex labeled 1 falls in $(v',v),$ since if it did, we would necessarily have $v'=\ast,$ but the only vertex $v$ for which $v^0(v)=\ast$ is the vertex labeled 1, but if $v$ is the vertex labeled 1, then $(v',v)=\emptyset.$  Therefore, any vertex of $(Wc)_{\tau(\w{e})}$ (resp. $(Wc)_{\iota(\w{e})}$) which is in $(v',v)$ must be blocked.  So, the edge $e'$ in the cell $(Wc)_{\tau(\w{e})}$ (resp. $(Wc)_{\iota(\w{e})}$) satisfies the conditions of Lemma \ref{lemma:BlockedInInterval}, so 
\[
\pi F^\infty\left((Wc)_{\tau(\w{e})}\right)=0=\pi F^\infty\left( (Wc)_{\iota(\w{e})}\right).\qedhere
\]
\end{proof}

Now, we prove an analogue of Lemma \ref{lemma:Rinfty} for the initial reduction $r_I:$
\begin{lemma}\label{lemma:rIinfty}

If $\Gamma$ is any graph with essential spanning tree $T$, and $c$ is a redundant cell in $\UD{n}{\Gamma}$ which consists exclusively of deleted edges and vertices,  we have $\pi F^\infty(c)=\pi F^\infty r_I(c).$
\end{lemma}

\begin{proof} 
If $c$ is defective, there is nothing to prove, so assume $c$ is non-defective, so $r_I(c)=r(c).$  Let $v$ be the minimal unblocked vertex in $c,$ so $Wc$ is obtained from $c$ by replacing $v$ with $e'=e^0(v).$ Consider the boundary
\[
\partial Wc=\sum_{e\in E(Wc)}\left((Wc)_{\iota(e)}+(Wc)_{\tau(e)}\right).
\]

If $e=e',$ then $(Wc)_{\iota(e)}=c$ and $(Wc)_{\tau(e)}=r(c).$ So, 
\begin{align*}
F(c)=(1+\partial W)(c)&=c+r(c)+c+\!\!\sum_{e\in E(c)}\!\!\left((Wc)_{\iota(e)}+(Wc)_{\tau(e)}\right)
=r(c)+\mathcal{C},
\end{align*}
where $\mathcal{C}=\sum_{e\in E(c)}\left((Wc)_{\iota(e)}+(Wc)_{\tau(e)}\right).$ By Corollary \ref{cor:DeletedEdgeEndpoints}, $\pi F^\infty(\mathcal{C})=0,$ so
\[
\pi F^\infty F(c)=\pi F^\infty(r(c))+\pi F^\infty(\mathcal{C})=\pi F^\infty (r(c)).
\]
Since $F^\infty F(c)=F^\infty(c),$ the claim follows.
\end{proof}

\begin{defn}
For any edge $e\subset \Gamma,$ let $\wedge e$ denote the maximal vertex on the intersection of the $T$-geodesics from $\tau(e)$ to $\ast$ and from $\iota(e)$ to $\ast.$
\end{defn}
Note that if $e$ is an edge in $T$, then $\wedge e=\tau(e).$  

\begin{cor}\label{cor:DeletedEdgeBoundary}
Let $\Gamma$ be any graph with an essential spanning tree $T$ and let $c$ be a critical cell in $\Gamma$ consisting exclusively of deleted edges and no vertices. If for each pair of edges $\w{e},\w{e}'$ in $c$ we have $[\wedge \w{e},\iota(\w{e})]\cap [\tau(\w{e}'),\iota(\w{e}')]=\emptyset,$
then $\w\partial c=0.$
\end{cor}
\begin{proof}
Recall the boundary $\w{\partial}$ is given by $\w{\partial}=\pi F^\infty\partial.$  For a cell $c$ as in the statement, the boundary $\partial c$ is of the form
\[
\partial c=\sum_{\w{e}\in E(c)}(c_{\tau(\w{e})}+c_{\iota(\w{e})}).
\]
Let $c'$ be a cell of the form $c_{\tau(\w{e})}$ or $c_{\iota(\w{e})}.$ If $\tau(\w{e})=\ast,$ and $c'=c_{\tau(\w{e})},$ then $c'$ is critical.  Otherwise, $c'$ is a redundant cell with minimal unblocked vertex $\tau(\w{e})$ or $\iota(\w{e}).$ Writing $r^0(c)=c$ and $r^j(c)=r(r^{j-1}(c))$ for $j>0,$ we see that there are integers $K$ and $L$ such that $r^K(c_{\tau(\w{e})})$ (resp. $r^L (c_{\iota(\w{e})})$) is obtained by moving $\tau(\w{e})$ (resp $\iota(\w{e})$) to $\wedge \w{e},$ so $r^K(c_{\tau(\w{e})})=r^L(c_{\iota(\w{e})}).$  

If $\tau(\w{e})\ne\ast,$ then the cells $r^k(c_{\tau(\w{e})})$ and $r^l(c_{\iota(\w{e})})$ for $k=0,\dots,K-1$ and $l=0,\dots,L-1$ are non-defective  by assumption, so $r_I^{k+1}(c_{\tau(\w{e})})=r^{k+1}(c_{\tau(\w{e})})$ and $r_I^{l+1}(c_{\iota(\w{e})})=r^{l+1}(c_{\iota(\w{e})})$ for each such $k$ and $l.$   It then follows from Lemma \ref{lemma:rIinfty} that 
\[
\pi F^\infty(c_{\tau(\w{e})})=\pi F^\infty(r^Kc_{\tau(\w{e})})=\pi F^\infty(r^Lc_{\iota(\w{e})})=\pi F^\infty(c_{\iota(\w{e})}).
\]
If $\tau(\w{e})=\ast,$ then $\pi F^\infty(c_{\tau(\w{e})})=c_{\tau(\w{e})}$ and $ r_I^Lc_{\iota(\w{e})}=c_{\tau(\w{e})}$ for some $L$ (since in this case, $\wedge\w{e}=\ast),$ so again $\pi F^\infty(c_{\tau(\w{e})})=\pi F^\infty(c_{\iota(\w{e})}).$
Therefore we have
\[
\w{\partial}c=\sum_{\w{e}\in E(c)}(\pi F^\infty c_{\tau(\w{e})}+\pi F^\infty c_{\iota(\w{e})})=0.\qedhere
\]
\end{proof}

The proof of Theorem \ref{thm:VertexDisjointUnordered} relies on a choice of spanning tree with certain properties.

\begin{lemma}\label{lemma:VertexDisjointSpanningTree}
Let $\nu=\nu(\Gamma),$ and let $\cycle_1,\dots,\cycle_\nu$ be vertex-disjoint cycles in $\Gamma.$  There exists an essential spanning tree $T\subset\Gamma$ and an embedding of $T$ in the plane with the following properties:
\begin{enumerate}
    \item Each cycle $\cycle_i$ contains exactly one deleted edge $\w{e}_i.$
    \item The edges $\w{e}_1,\dots,\w{e}_\nu$ satisfy $[\wedge\w{e}_i,\iota(\w{e}_i)]\cap
    [\wedge\w{e}_j,\iota(\w{e}_j)]=\emptyset$ for $i\ne j.$
\end{enumerate}
\end{lemma}

\begin{proof}
For each $i=1,\dots,\nu,$ choose an arbitrary edge $e'_i$ on $\cycle_i.$  We construct a spanning tree $T'$ by adding edges of $\Gamma$ as follows (starting with no edges).  First, we add each edge on each $\cycle_i$ except $e'_i.$  Then, chose an arbitrary ordering of the edges of $\Gamma$ which do not fall on any $\cycle_i.$  Inductively, add each edge if and only if it does not form a cycle.  This describes a spanning tree $T'$ with the property that each $\cycle_i$ contains exactly one deleted edge (the edge $e'_i$).  By subdividing if necessary, we can assume the endpoints of each $e'_i$ have degree 2 in $\Gamma,$ so that $T'$ is an essential spanning tree.

Since the cycles $\cycle_1,\dots,\cycle_\nu$ are vertex-disjoint, the graph $\Gamma'=T'\cup\bigcup e'_i$ does not contain a subgraph homeomorphic to either the complete graph $K_5$ or the complete bipartite graph $K_{3,3},$ so by Kuratowski's Theorem, the graph $\Gamma'$ is planar (see \cite{ChartrandLesniakZhang}, for example). It is clear that we may chose an embedding of $\Gamma'$ into the plane with the property that the interior of the region of the plane bounded by $\cycle_i$ does not contain any vertices of $\Gamma'.$  Choose such an embedding, and let $\ast$ be a vertex of degree 1 in $T'.$  If $\ast$ falls on some deleted edge $e'_i,$ we require it is chosen so that the path in $T'$ from $\ast=\tau(e'_i)$ to $\iota(e'_i)$ is travelled counterclockwise along $\cycle_i.$  With this choice, the vertices along this path form an interval $[\ast,\iota(e'_i)].$  See Figure \ref{fig:VertexDisjoint}.  Three vertex-disjoint cycles have been highlighted with bold lines.
\begin{figure}[htb]
    \centering
    \includegraphics[height=7cm]{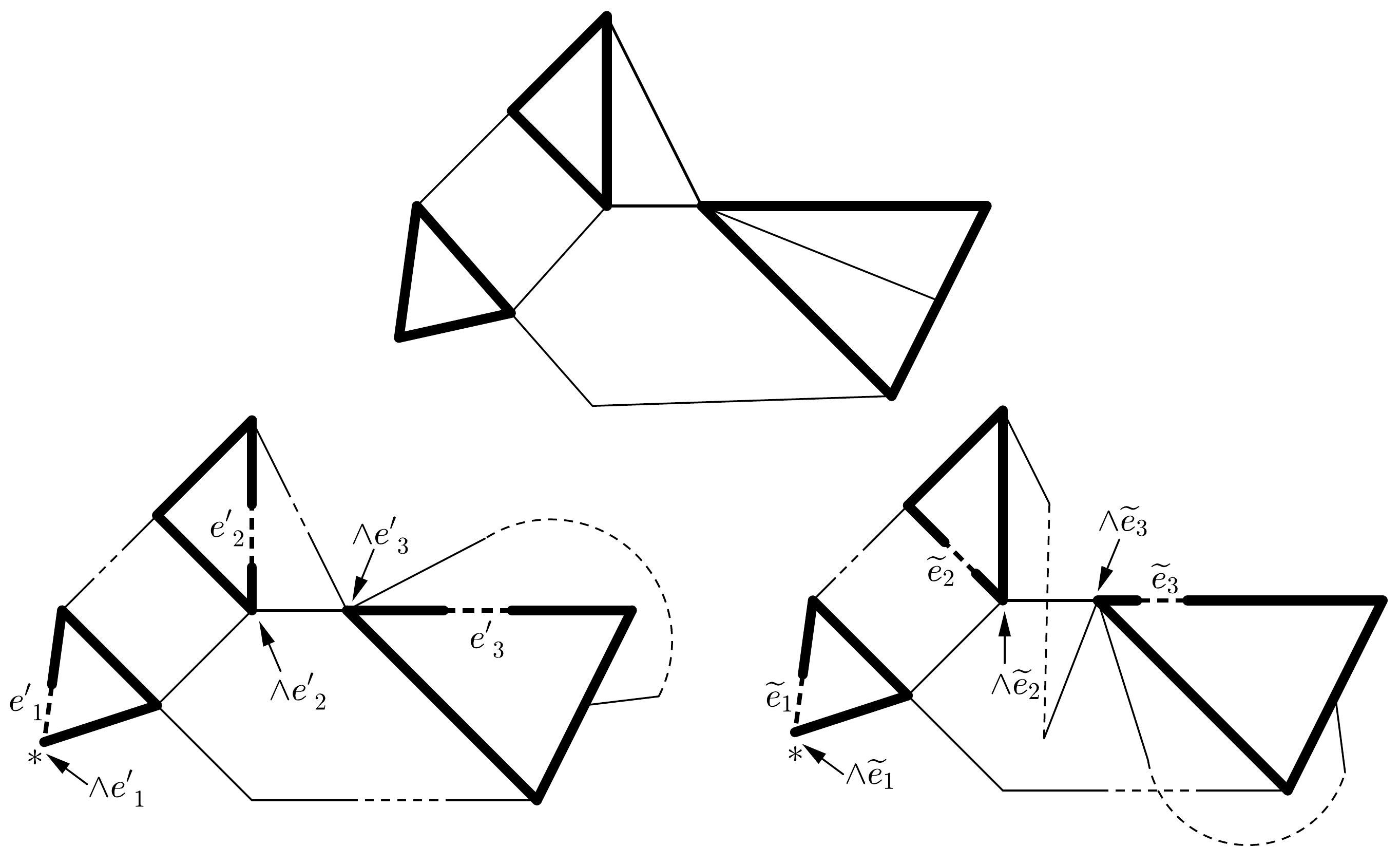}
    \caption{A graph $\Gamma$ (top); a choice of spanning tree $T'$ (bottom left); a choice of the spanning tree $T$ (bottom right)}
    \label{fig:VertexDisjoint}
\end{figure}

For each edge $e'_i$ which doesn't have $\ast$ as an endpoint, the vertex $\wedge e'_i$ is neither $\tau(e'_i)$ nor $\iota(e'_i),$ and it is clear that $\wedge e'_i$ must fall on $\cycle_i.$  We modify the embedding so that the two edges on $\cycle_i$ incident to $\wedge e'_i$ fall in directions 1 and 2 from $\wedge e'_i.$  Finally, we modify the spanning tree $T'$ by adding the edge $e'_i$ and removing the edge $\w{e}_i$ on $\cycle_i$ which has the property that $\tau(\w{e}_i)=v^1(\wedge e'_i).$  Again by subdividing if necessary, we can assume the endpoints of each $\w{e}_i$ have degree 2 in $\Gamma.$  This describes the spanning tree $T.$  By the choice of $T',$ each $\cycle_i$ contains exactly one deleted edge (the edge $\w{e}_i$).  By the choice of the embedding of $T$ and the edges $\w{e}_i,$ the vertices on the cycle $\cycle_i$ form an interval $[\wedge\w{e}_i,\iota(\w{e}_i)].$  Since the cycles are disjoint, these intervals are also disjoint.  Finally, we subdivide $T$ so that it is sufficiently subdivided; this has no effect on the property that the intervals $[\wedge\w{e}_i,\iota(\w{e}_i)]$ are disjoint.
\end{proof}

Now we describe the cells that satisfy the properties in Lemma \ref{lemma:DistinctHomologyClasses}.
\begin{defn}
Let $\Gamma$ be a graph with $\nu=\nu(\Gamma)\ge1.$  Choose a spanning tree $T$ and deleted edges $\w{e}_1,\dots,\w{e}_\nu$ as in Lemma \ref{lemma:VertexDisjointSpanningTree}. Let $n$ be an integer satisfying $1\le n\le\nu$ and consider a subset $R$ of $\{1,\dots,\nu\}$ with $|R|=n.$  Let $\Psi_{R}(\Gamma)$ denote the critical $n$-cell in $\UD{n}{\Gamma}$ consisting exclusively of the deleted edges $\w{e}_i$ for each $i\in R.$
\end{defn}

\begin{lemma}
The cells $\Psi_R(\Gamma)$ represent distinct classes in $H_*^{\mathcal{M}}(\UD{n}{\Gamma}).$
\end{lemma}

\begin{proof}
By Corollary \ref{cor:DeletedEdgeBoundary}, we have $\widetilde{\partial}(\Psi_R)=0.$  The lack of $(n+1)$-cells shows that each $\Psi_R$ represents a distinct homology class.
\end{proof}

\begin{lemma}
The equivalence class $[\Psi_R(\Gamma)]$ contains only the cell $\Psi_R(\Gamma),$ and is the least upper bound for the unique collection of equivalence classes of $n$ 1-cells which have $[\Psi_R(\Gamma)]$ as its least upper bound.
\end{lemma}

\begin{proof}
The first statement follows immediately from the fact that $\Psi_R(\Gamma)$ does not contain any vertices.  Furthermore, any equivalence class of $n$-cells which is an upper bound for the unique collection of equivalence classes of $n$ 1-cells which has $[\Psi_R(\Gamma)]$ as an upper bound must contain the $n$ edges $\w{e}_i,$ $i\in R,$ but $\Psi_R(\Gamma)$ is the only such cell in $\UD{n}{\Gamma}.$
\end{proof}

\begin{proof}[Proof of Theorem \ref{thm:VertexDisjointUnordered}]
The upper bound $\TC(\UD{n}{\Gamma})\le 2n+1$ is given in Theorem \ref{thm:TCUDnUB}, and to establish the lower bound, we will use Theorem \ref{thm:TCTools}. Let $R_1=\{1,2,\dots,n\}$ and $R_2=\{n+1,n+2,\dots,2n\},$ and let $\Psi_1=\Psi_{R_1}(\Gamma)$ and $\Psi_2=\Psi_{R_2}(\Gamma).$  Let $\{[c_1],\dots,[c_n]\}$ and $\{[c_{n+1}],\dots,[c_{2n}]\}$ be the unique collections of equivalence classes of 1-cells which have $[\Psi_1]$ and $[\Psi_2]$ as their least upper bounds.  Consider the product of zero-divisors
\begin{align}
\prod_{i=1}^{2n}\overline{\phi}_{[c_i]}=\prod_{i=1}^{2n}(\phi_{[c_i]}\otimes 1+1\otimes \phi_{[c_i]}).\label{eqn:ZDPsi}
\end{align}
Since the top dimension in cohomology is $n,$ this product is a sum of elements of the form 
\begin{equation}
\phi_{[c_{i_1}]}\cdots\phi_{[c_{i_n}]}\otimes\phi_{[c_{i_{n+1}}]}\cdots\phi_{[c_{i_{2n}}]},
\label{eqn:ZDPsib}
\end{equation}
where $\{i_1,\dots,i_{2n}\}=\{1,\dots,2n\}.$  If 
$
I=\{i_1,\dots,i_{n}\}$ and $J=\{i_{n+1},\dots,i_{2n}\},
$
then $[\Psi_I(\Gamma)]$ and $[\Psi_J(\Gamma)]$ are the least upper bounds for the collections  $\{[c_{i_1}],\dots,[c_{i_n}]\}$ and $\{[c_{i_{n+1}}],\dots,[c_{i_{2n}}]\}$ respectively, so the term in (\ref{eqn:ZDPsib}) equals $\phi_{[\Psi_I(\Gamma)]}\otimes\phi_{[\Psi_J(\Gamma)]}.$  Since the cells $\Psi_R(\Gamma)$ all represent distinct homology classes, we see the product (\ref{eqn:ZDPsi}) is non-zero, and the lower bounds then follow from Theorem \ref{thm:TCTools}.  
\end{proof}

We conclude with the following corollary regarding complete graphs.
\begin{cor}
Let $K_m$ denote the complete graph on $m\ge6$ vertices, and let $n$ be an integer satisfying $1\le n\le\left\lfloor\frac{m}{6}\right\rfloor.$  Then, $\TC(\UD{n}{\Gamma})=2n+1.$
\end{cor}
\begin{proof}
This follows from Theorem \ref{thm:VertexDisjointUnordered} and the fact that $\nu(K_m)=\left\lfloor\frac{m}{3}\right\rfloor.$
\end{proof}

\bibliographystyle{plain}
\bibliography{TCUnorderedGraphConfigs}

\end{document}